\documentclass[a4, 12pt, reqno]{amsart}

\pdfoutput=1 

\usepackage[latin1]{inputenc}
\usepackage{amsxtra}
\usepackage{amsmath}
\usepackage{amscd}
\usepackage{amssymb}
\usepackage{amsfonts}
\usepackage[all]{xy}
\usepackage{mathrsfs}
\usepackage{amsthm}
\usepackage{color}
\usepackage{upgreek}
\usepackage{verbatim}
\usepackage{enumerate}
\usepackage{latexsym}
\usepackage{graphicx}
\usepackage{MnSymbol}
\usepackage{todonotes}
\usepackage{setspace}
\usepackage{hyperref}
\usepackage{stmaryrd}
\usepackage{nicefrac}
\usepackage{euscript}

\usepackage{amsthm,amsbsy,amsfonts,amssymb,amscd}
\usepackage{tikz}
\usetikzlibrary{calc,intersections,through, backgrounds,arrows,decorations.pathmorphing,backgrounds,positioning,fit,petri} \usetikzlibrary{calc} \usepackage{tikz-cd} \usepackage[all]{xy} \xyoption{knot} \xyoption{poly}
\usepackage{mathtools} 
\usepackage{hyperref} 
\usepackage[margin=1in]{geometry} \usepackage{etoolbox}\usepackage{ytableau}\usepackage{youngtab}\makeatletter\pretocmd{\appendix}{\addtocontents{toc}{\protect\addvspace{10\p@}}}{}{}\makeatother

\usetikzlibrary{decorations.pathmorphing}

 \definecolor{darkgreen}{HTML}{336633}
 \definecolor{darkred}{HTML}{993333}
\makeatletter

\theoremstyle{plain}
\newtheorem{thm}{Theorem}[section]
\newtheorem{lemma}[thm]{Lemma}
\newtheorem{prop}[thm]{Proposition}

\newtheorem{cor}[thm]{Corollary}
\newtheorem{df-prop}[thm]{Definition-Proposition}

\theoremstyle{definition}
\newtheorem{df}[thm]{Definition}

\theoremstyle{remark}
\newtheorem{rem}[thm]{Remark}

\xyoption{all}
\def\frakg{\mathfrak{g}}

\def\frakp{\mathfrak{p}}

\def\vep{\varepsilon}

\def\End{\operatorname{End}\nolimits}

\def\tr{\operatorname{tr}\nolimits}

\def\frakgl{\mathfrak{gl}}
\def\gr{\operatorname{gr}}

\def\pn{\frakp (n)}
\def\ov{\overline}

\newcommand{\nc}{\newcommand}

\nc{\mf}{\mathfrak}
\nc{\mc}{\mathcal}
\nc{\VWd}{{\bigdoublevee}_d^{\mf p(n)}}
\nc{\Z}{{\mathbb Z}}
\nc{\C}{{\mathbb C}}
\nc{\N}{{\mathbb N}}
\nc{\bara} {{\bar{a}}}
\nc{\barb} {{\bar{b}}}
\nc{\bari} {{\bar{i}}}
\nc{\barj} {{\bar{j}}}
\nc{\barell} {{\bar{\ell}}}
\nc{\hZ}{{\underline{\mathbb Z}}}
\nc{\F}{{\mf F}}
\nc{\Q}{\mathbb {Q}}
\nc{\ep}{\delta}
\nc{\h}{\mathfrak h}
\nc{\n}{\mf n}
\nc{\A}{{\mf a}}
\nc{\G}{{\mathfrak g}}
\nc{\SG}{\overline{\mathfrak g}}
\nc{\DG}{\widetilde{\mathfrak g}}
\nc{\D}{\mc D}
\nc{\Li}{{\mc L}}
\nc{\La}{\Lambda}
\nc{\is}{{\mathbf i}}
\nc{\V}{\mf V}
\nc{\bi}{\bibitem}
\nc{\NS}{\mf N}
\nc{\dt}{\mathord{\hbox{${\frac{d}{d t}}$}}}
\nc{\E}{\EE}
\nc{\ba}{\tilde{\pa}}
\nc{\half}{\frac{1}{2}}
\nc{\hgl}{\widehat{\mathfrak{gl}}}
\nc{\gl}{{\mathfrak{gl}}}
\nc{\hz}{\hf+\Z}
\nc{\dinfty}{{\infty\vert\infty}}
\nc{\SLa}{\overline{\Lambda}}
\nc{\SF}{\overline{\mathfrak F}}
\nc{\SP}{\overline{\mathcal P}}
\nc{\U}{\mathfrak u}
\nc{\SU}{\overline{\mathfrak u}}
\nc{\wh}{\widehat}
\nc{\sL}{\ov{\mf{l}}}
\nc{\sP}{\ov{\mf{p}}}
\nc{\osp}{\mf{osp}}
\nc{\spo}{\mf{spo}}
\nc{\hosp}{\widehat{\mf{osp}}}
\nc{\hspo}{\widehat{\mf{spo}}}
\nc{\hh}{\widehat{\mf{h}}}
\nc{\even}{{\bar 0}}
\nc{\odd}{{\bar 1}}
\nc{\mscr}{\mathscr}
\nc{\Pd}{{\widehat{P}_d^-}}

\nc{\wti}{\widetilde}
\title[Affine Periplectic Brauer Algebras]{Affine Periplectic Brauer Algebras}

\author[Chen]{Chih-Whi Chen}
\address{Mathematics Division, National Center for Theoretical Sciences, Taipei, Taiwan 10617} \email{d00221002@ntu.edu.tw}

\author[Peng]{Yung-Ning Peng}
\address{Department of Mathematics, National Central University,
Chung-Li, Taiwan, 32054} \email{ynp@math.ncu.edu.tw}

\begin{document}
\begin{abstract}
  We formulate Nazarov-Wenzl type algebras $\Pd$ for the representation theory of the periplectic Lie superalgebras $\mf p(n)$.  We establish an Arakawa-Suzuki type functor to provide a connection between $\mf p(n)$-representations and $\Pd$-representations. We also consider various tensor product representations for $\Pd$.  The periplectic Brauer algebra $A_d$ developed by Moon is a quotient of $\Pd$. In particular, actions induced by Jucys-Murphy elements can also be recovered under the tensor product representation of $\Pd$. Moreover, a Poincare-Birkhoff-Witt type basis for $ \widehat{P}_d^-$ is obtained. A diagram realization of $\Pd$ is also obtained.
\end{abstract}

\numberwithin{equation}{section}

\maketitle


\let\thefootnote\relax\footnotetext{{\em Key words:} Periplectic Lie
superalgebras, marked Brauer category, periplectic Brauer algebra, affine Nazarov-Wenzl algebra, Arakawa-Suzuki type functor, Jucys-Murphy elements, affine periplectic Brauer algebra, Poincare-Birkhoff-Witt basis.}

\section{Introduction}
\subsection{}The  periplectic Lie superalgebra $\mf p(n)$ is a superanalogue of
the orthogonal or symplectic Lie algebra preserving an odd non-degenerate symmetric or skew-symmetric bilinear form. In the past three decades, there have been some related studies on $\mf p(n)$-representation theory, see, e.g., \cite{Sc}, \cite{PS}, \cite{Go}, \cite{Se1}, \cite{Ch}, \cite{Co}, \cite{B+9}, \cite{CE1} and \cite{CE2}. In particular, the finite-dimensional irreducible character problem has been solved in  \cite{B+9}. However, the representation theory of the periplectic Lie superalgebra is still not well-understood. One of the main reasons is that many classical and traditional methods in representation theory are not
applicable. In particular, the center of its universal enveloping algebra fails to provide
us with information about the blocks in the respective categories (cf. \cite{Go}).

\subsection{} In recent years, diagrammatically defined algebras have naturally appeared in numerous areas of mathematics. In particular,  their representation theory have recently aroused much interest. The very first exposition was studied by Brauer in his celebrated paper \cite{Br} where the Brauer algebras were formulated as  centralizers:
\begin{align*} Br_d(\delta) \longrightarrow \text{End}_G(V^{\otimes d}),
\end{align*} where $G \subset GL(V)$ is the orthogonal or symplectic group. The method of establishing a link between representation theories via diagram algebras has attracted considerable attention and offers new perspectives in representation theory.

 Brauer's theory has been known to admit a generalization in the $\Z_2$-graded setting (see, e.g., \cite{BSR}). In particular, Moon introduced and studied a Brauer type algebra $A_d$ in \cite{Mo} for the first time by giving generators and relations of $A_d$. Furthermore, the connections between tensor product representations of $\mf p(n)$ and $A_d$ were established. Moon used  Bergman's diamond Lemma to prove that the dimension of $A_d$ is identical to the Brauer algebra $Br_d(0)$ and noticed that the generators and relations of $A_d$ bear resemblance to $Br_d(0)$. Later on, Kujawa and Tharp provided certain diagrammatically defined algebras in \cite{KT}, which gave a uniform method to study the algebras defined by Brauer and Moon simultaneously.

In a recent article \cite{Co}, Coulembier studied the {\em periplectic Brauer algebra}, which is exactly the algebra $A_d$ discovered by Moon, for the invariant theory for $\mf p(n)$. As an application, the blocks in the category of finite dimensional weight modules over $\mf p(n)$ have been determined.

\subsection{} The {\em degenerate affine Nazarov-Wenzl algebra} $We(d)$, introduced in \cite[Section 4]{Na}, can be regarded as an affine analogue of the Brauer algebra $Br_d(0)$. In particular, $Br_d(0)$ is a homomorphic image of $We(d)$, where the Jucys-Murphy elements in $Br_d(0)$ are images of the polynomial generators in $We(d)$.

The main purpose of this paper is to study the affine version of periplectic Brauer algebras. We first formulate the definition of the {\em affine periplectic Brauer algebra} $\Pd$ by generators and relations, together with an action of $\Pd$ on $M\otimes V^{\otimes d}$, where $M$ is an arbitrary $\mf p(n)$-module and $V=\C^{n|n}$.
In particular, the construction allows us to define an Arakawa-Suzuki type functor $\mc F$ from the category of $\mf p(n)$-modules to the category of $\Pd$-modules; a similar approach appears also in \cite[Section 4]{B+9}, where the properties and applications of similar functors are studied.
In the second part, we establish a homomorphism from $\Pd$ to the periplectic Brauer algebra $A_d$. In particular, the Jucys-Murphy elements of $A_d$ given in \cite{Co} are precisely the images of certain polynomial generators of $\Pd$. Moreover, we give a Poincare-Birkhoff-Witt type basis for $\Pd$.
Finally we give a diagrammatic interpretation for $\Pd$, extending the result of \cite{Mo} and partially of \cite{KT}.

\subsection{}This paper is organized as follows. In Section~\ref{Section::glnnandpn}, some basic definition and notation for Lie superalgebras $\mf{gl}(n|n)$ and $\mf p(n)$ are introduced. Subsequently, we give an explicit construction of a Casimir-like element $C$.
The definition of the affine periplectic Brauer algebra $\Pd$ and a $\mf p(n)$-analogue of Arakawa-Suzuki type functor $\mc F$ are  established in Section~\ref{Section::AffinePBrandASThm}. Particularly, we evaluate the functor $\mc F$ at various $\mf p(n)$-modules in Section~\ref{Section::Tensors}. A Poincare-Birkhoff-Witt type basis for $\Pd$ is given in Section~\ref{Section::PBWbasis}. Finally, a diagram realization of $\Pd$ is established in Section~\ref{Section::diagram}.


\subsection{}
Throughout the paper, the symbols $\Z$, $\N$, and $\Z_+$ stand for the sets of all,
positive and non-negative integers, respectively. All vector spaces, algebras, tensor
products, are over the field of complex numbers $\C$.  For given non-negative integers $m,n$, we define a partial ordered set $I(m|n):= \{1<2<\cdots <m< \ov 1 <\ov 2 <\cdots <\ov n\}$.
For  a superspace $M = M_{\overline{0}}\oplus M_{\overline{1}}$ and a given homogenous element $m\in M_{\chi}$ ($\chi\in \mathbb{Z}_2$), we let $\ov{m}$= $\chi$.

\vskip 1cm
{\bf Acknowledgement.}
We are grateful to Weiqiang Wang and Shun-Jen Cheng for their suggestion which triggered this collaboration, and for the crucial questions/advice they have given during this project.
We thank Yongjie Wang for providing us with the references \cite{KT, Mo}.
Peng is partially supported by MOST grant 105-2628-M-008-004-MY4 and NCTS Young Theorist Award 2015.  At almost the same time that we posted the first version of the present article, Balagovic et al also posted \cite{B+9} on the arxiv, which contains results overlapping with ours. We are grateful to Jonathan Kujawa and Vera Serganova for communication.

\section{Lie superalgebras $\mf{gl}(m|n)$ and $\pn$} \label{Section::glnnandpn}

\subsection{Basic setting and notation for $\gl(n|n)$ and $\mf p(n)$}
Let $m,n\in\Z_+$. Let $\mathbb{C}^{m|n}$ be the complex superspace of superdimension $(m|n)$. The {\em general linear Lie superalgebra} $\mathfrak{gl}(m|n)$ may be realized as $(m+n) \times (m+n)$ complex matrices:
\begin{align} \label{glrealization}
 X=\left( \begin{array}{cc} A & B\\
C & D\\
\end{array} \right),
\end{align}
where $A,B,C$ and $D$ are respectively $m\times m, m\times n, n\times m, n\times n$ matrices. Let $\text{Str}: \mf{gl}(m|n) \rightarrow \C$ denote the supertrace function given by $\text{Str}(X)=\tr A-\tr D$.

In the rest of this paper, we fix $m=n$ and let $I:= I(n|n) = \{1,2, \cdots ,n ,\overline{1}, \overline{2},\cdots ,\overline{n}\}$ and let $I^0 := \{1,2, \cdots, n\}$.
Let $\{e_i\}_{i\in I}$ be the basis of the standard $\gl(n|n)$-representation $V:= \C^{n|n}$, where $(\mathbb{C}^{n|n})_{\bar{0}} =\sum_{i\in I^{0}}\mathbb{C}e_i$ and $(\mathbb{C}^{n|n})_{\bar{1}} =\sum_{i\in I \setminus I^{0}}\mathbb{C}e_i$. For each $\overline{i} \in I\backslash I^0$, we set $\overline{\overline{i}} =i$.   We equip $V$  with the odd bilinear form $(\cdot, \cdot)$ given by $(e_a, e_b) = \delta_{a,\ov b}$, for all $a,b \in I$.

Throughout this paper, let
$$\mathfrak{g}=\mathfrak{g}_{\overline{0}} \oplus
\mathfrak{g}_{\overline{1}}$$  denote the {\it periplectic Lie
superalgebra} $\mf p(n)$ defined in \cite{Ka}. Recall that $\mathfrak{g}$
admits a $\mathbb{Z}$-gradation
$\mathfrak{g}=\mathfrak{g}_{-1} \oplus \mathfrak{g}_0 \oplus
\mathfrak{g}_{+1}$ where $\mathfrak{g}_0 =
\mathfrak{g}_{\overline{0}} =\mathfrak{gl}(n)$,
$\mathfrak{g}_{-1} \cong \bigwedge^2(\mathbb{C}^{n*})$,
and $\mathfrak{g}_{+1} \cong S^{2}(\mathbb{C}^n)$ as
$\mathfrak{g}_0$-modules.
The standard matrix realization is given
by
\[\mf p(n)=
\left\{ \left( \begin{array}{cc} A & U\\
L & -A^t\\
\end{array} \right),\text{where $U$ is symmetric and $L$ is skew-symmetric} \right\} \subset  \mf{gl}(n|n).
\]
It is well-known that $\gl(n|n)\cong \frak{g}\oplus \frak{g}^*$ as $\G$-modules, and the $\mathbb{Z}$-gradation of $\mf p(n)$ is compatible to the $\mathbb{Z}_2$-grading of $\gl(n|n)$.

\subsection{A Casimir-like element}
As explained in the introduction of \cite{B+9}, one of the difficulties in the study of the representation theory of $\mf p(n)$ is the lack of quadratic Casimir elements in $U(\mf p(n))$. However, a key observation is that the embedding $\mf p(n)\subset \gl(n|n)\cong \mf g\oplus \mf g^*$ allows one to construct a Casimir-like element, see \cite[Section 4.1]{B+9}. As a result, some of the classical approaches can be applied after suitable modifications if necessary.

\begin{prop} \emph{(}\cite[Lemma 4.1.2]{B+9}, \cite{Co} \emph{)}\label{Prop::SC}
Let $\{x_i\} \cup \{x_i^*\} \subset \mf{gl}(n|n)$ be homogenous elements such that $\{x_i\}$ is a basis for $\frak{g}$ and $\{x_i^*\}$ is the dual basis for $\frak{g}^*$ with respect to the supertrace form given by $\emph{Str}(x_i^*x_j)=\delta_{ij}$. Let
 \begin{align} \label{ComesElement}
 C =\sum_{i=1}^{\emph{dim }\G} x_i \otimes x_i^* \in \frakg \otimes \frakgl(n|n).
 \end{align} Then for each $\mf g$-module $M$, we have
  $C\in \emph{End}_{\mf g}(M\otimes V)$.
\end{prop}
\begin{proof}
Set $\ell =\dim \G$. For each $1\leq a,b\leq \ell$, write $[x_a,x_b] = \sum_{q=1}^{\ell}C_{ab}^q x_q$. It is a well-known fact that the supertrace form is an even supersymmetric invariant bilinear form on $\gl(n|n)$, and hence we have $\ov{x_i} = \ov{x_i^*}$, $\forall 1\leq i\leq \ell$.

Since $\text{Str}(x_k\sum_{q=1}^{\ell}C_{qa}^b(-1)^{\overline{x_q}} x_{q}^*) =C_{ka}^b  = (-1)^{\overline{x_b^*}}\text{Str}([x_k,x_a]x_{b}^*) =(-1)^{\overline{x_b^*}}\text{Str}(x_k[x_a,x_{b}^*]) $ for all $1\leq a,b,k\leq \ell$, we have that $[x_a, x_{b}^*] = \sum_{q=1}^{\ell}C_{qa}^b(-1)^{\overline{x_b^*}+\overline{x_q}} x_{q}^*$, for all $1\leq a,b\leq \ell$. Now fix $1\leq k\leq \ell$ and homogenous elements $m\in M$, $v\in V$. By using the fact that $C_{ij}^{q} =(-1)^{1+\overline{x_i}\cdot\overline{ x_j}} C_{ji}^q$, together with the following fact, \begin{align} &C_{ij}^{q} \neq 0 \text{ implies } \overline{x_i }+\overline{x_j }=\overline{x_q}, \label{ComesElement::NonZeroCoeff}\end{align} we have
\begin{align*}
  &C(x_k(m\otimes v))) = (\sum_{i=1}^{\ell}x_i\otimes x_i^*)  (x_km\otimes v + (-1)^{\overline{x_k}\cdot \overline{m}}m\otimes x_kv) \\ &=  \sum_{i=1}^{\ell}  \{((-1)^{(\overline{x_k}+\overline{m}) \overline{x_i^*}}[x_i,x_k]m \otimes x_i^*v)+ ((-1)^{\overline{x_k}\overline{m}+\overline{x_i^*} \cdot \overline{m}}x_im \otimes[x_i^*,x_k]v)\}  +  x_k(C(m\otimes v))\\
  &=  \sum_{i,j=1}^{\ell}  \{((-1)^{ (\overline{x_k}+\overline{m}) \overline{x_j^*}}C^{i}_{jk} x_im\otimes x_j^*v)+ ((-1)^{\overline{x_k}\overline{m}+\overline{x_i^*} \cdot \overline{m}+\overline{x_i^*} \cdot \overline{x_k} +1 + \overline{x_j^*} + \overline{x_i^*}} C_{jk}^{i}x_im \otimes x_j^*v)\}  +  x_k(C(m\otimes v))\\
  &=  \sum_{i,j=1}^{\ell}  \{((-1)^{ (\overline{x_k}+\overline{m}) \overline{x_j^*}}C^{i}_{jk}+ (-1)^{\overline{x_k}\overline{m}+\overline{x_i^*} \cdot \overline{m}+\overline{x_i^*} \cdot \overline{x_k} +1 + \overline{x_j^*} + \overline{x_i^*}} C_{jk}^{i})(x_im \otimes x_j^*v)\}  +  x_k(C(m\otimes v)) \\ &= x_k(C(m\otimes v)).
\end{align*}
The last equality follows from the following calculation:
If $C_{jk}^{i} \neq 0$ then by \eqref{ComesElement::NonZeroCoeff} we have that
\begin{align*}
  &(\overline{x_k}+\overline{m}) \overline{x_j^*} \equiv    (\overline{x_k}+\overline{m}) (\overline{x_i}+ \overline{x_k}) \equiv \overline{x_k}+ \overline{x_i^*}  \overline{x_k}+ \overline{x_k}\overline{m}+\overline{x_i^*}  \overline{m}  \\   &\equiv  \overline{x_j} + \overline{x_i}  +\overline{x_i^*}  \overline{x_k} + \overline{x_k}\overline{m}+\overline{x_i^*}  \overline{m} \text{ (mod 2)}.
\end{align*}
\end{proof}

\begin{rem}
	The Casimir-like element for the queer Lie superalgebra $\mathfrak{q}(n)$ appeared in \cite{HKS}, which was inspired by \cite{Ol}. The Casimir-like element in Proposition \ref{Prop::SC} is due to Dimitar Grantcharov and Jonathan Kujawa. We have since learned that the existence of this element was also known to Dimitar Grantcharov and Jonathan Kujawa.
\end{rem}

We decompose $\frakgl(n|n)$ into a direct sum of $\mathfrak{g}$-submodules $\mathfrak{g}$ and $\mathfrak{g}^*$ as follows:
 \begin{align*}
 & \frakgl(n|n) =\mathfrak{g} \oplus \mathfrak{g}^*  = \mathfrak{g}\oplus
\left\{ \left( \begin{array}{cc} A & J\\
K & A^t\\
\end{array} \right) ,
\text{where $J$ is skew-symmetric and $K$ is symmetric } \right\}.
 \end{align*}
 Let $\ell=\dim \G$. A choice of basis $\{x_i\}$ for $\G$ and its dual basis $\{x_i^*\}$ for $\G^*$ is given as follows with respect to the supertrace form on $\mf{gl}(n|n)$ given by $\langle x_i, x^*_j\rangle := \text{Str}(x^*_jx_i)$, for all $1\leq i,j \leq \ell$.
 \begin{align*}
 \{x_i\}_{i=1}^{\ell}:=\left\{ E_{st}:= \left( \begin{array}{cc} e_{st} & 0\\
0 & -e_{ts}\\
\end{array} \right) \right\} \bigcup
\left\{ X_{st} :=\left( \begin{array}{cc}  0 & e_{st}+e_{ts}\\
0 & 0 \\
\end{array} \right) \right\} \bigcup \\
\left\{X_{ss}:= \left( \begin{array}{cc}  0 & e_{ss}\\
0 & 0 \\
\end{array} \right) \right\} \bigcup
 \left\{ Y_{st}:=\left( \begin{array}{cc} 0 & 0\\
e_{st}-e_{ts} & 0\\
\end{array} \right) \right\},
 \end{align*}
 \begin{align*}
 \{x_j^*\}_{j=1}^{\ell}:=\left\{ E_{st}^*:=\frac{1}{2} \left( \begin{array}{cc} e_{ts} & 0\\
0 & e_{st}\\
\end{array} \right) \right\} \bigcup
\left\{ X_{st}^*:= -\frac{1}{2} \left( \begin{array}{cc}  0 &0\\
 e_{ts}+e_{st} & 0 \\
\end{array} \right) \right\} \bigcup \\
\left\{ X_{ss}^*:=-\left( \begin{array}{cc}  0 & 0\\
e_{ss} & 0 \\
\end{array} \right) \right\} \bigcup
 \left\{ Y_{st}^*:= - \frac{1}{2} \left( \begin{array}{cc} 0 & e_{st}-e_{ts} \\
0 & 0\\
\end{array} \right) \right\},
 \end{align*} where $1\leq s\neq t \leq n$. Here $e_{st}$ denotes the elementary $n\times n$ matrix with $(s,t)$-entry $1$ and zero elsewhere. Throughout this article, we shall consider the element $C$ with respect to the above choice of bases.

\vskip 1cm

\section{Affine Periplectic Brauer algebras} \label{Section::AffinePBrandASThm}

 \subsection{Generators and Relations}
 In this section, we give the definition of the affine periplectic Brauer algebra $\Pd$, which is the main object studied in this article.
It can be regarded as an affine analogue of the periplectic Brauer algebra $A_d$ discussed in \cite{Co, Mo} (also known as the marked Brauer algebra in \cite{KT}), or a periplectic analogue of the degenerate affine Nazarov-Wenzel algebra $We(d)$ studied in \cite{ES, Na}.

\begin{df}
\label{def:VW}
Let $d \in \N$. The {\em affine periplectic Brauer algebra}, denoted by $\Pd$, is the associative algebra over $\C$ generated by the elements
\begin{eqnarray*}
 s_a, \,\,   \vep_a, \,\,   y_j, & \quad &1 \leq a \leq d-1, \,\,1 \leq j \leq d,
\end{eqnarray*}
subject to the relations (for all $1
\leq a,b,c \leq d-1$ and $1 \leq i,j \leq d$):
\begin{enumerate}[($ P$.1)]
\item $ s_a^2 = 1$;\label{df1} 
\item 	
\begin{enumerate} 	
\item $ s_a s_b =  s_b s_a$ for $\mid a-b \mid > 1$; 	 \label{df2a}
\item $ s_c  s_{c+1} s_c = s_{c+1} s_c s_{c+1}$; 	 \label{df2b}
    \item $ s_a y_i = y_i s_a$ for $i \not\in \{a,a+1\}$; \label{df2c}
    \end{enumerate}
\item $ \vep_a^2 =0$; \label{df3}
\item $ \vep_1y_1^k \vep_1 = 0$ for $k \in \N$; \label{df4}
\item 	
\begin{enumerate}
\item $ s_a \vep_b =  \vep_b s_a$ and $ \vep_a \vep_b =  \vep_b \vep_a$ for $\mid a-b \mid > 1$; \label{df5a}	
\item $ \vep_ay_i =  y_i \vep_a$ for $i \not\in \{a,a+1\}$; \label{df5b}	
\item $ y_i y_j =  y_j y_i$; \label{df5c}	
\end{enumerate}
\item
\begin{enumerate} 	
\item $ \vep_as_a = - \vep_a = - s_a \vep_a$; 	\label{df6a}
\item $ s_c \vep_{c+1} \vep_c = - s_{c+1}\vep_c$ and
    $ \vep_c \vep_{c+1} s_c =  \vep_c s_{c+1}$; \label{df6b}	
\item $ \vep_{c+1} \vep_cs_{c+1} = - \vep_{c+1} s_c$ and $ s_{c+1} \vep_c \vep_{c+1} =  s_c \vep_{c+1}$; \label{df6c}	
\item $ \vep_{c+1} \vep_c \vep_{c+1} =
    -\vep_{c+1}$ and $\vep_c\vep_{c+1}\vep_c = -\vep_c$; 	\label{df6d}
\end{enumerate}
\item $ s_a y_a -  y_{a+1} s_a =  \vep_a - 1$ and $ y_a s_a - s_a y_{a+1} = - \vep_a - 1$; \label{df7}
\item 	
\begin{enumerate} 	
    \item $ \vep_a( y_a-y_{a+1}) =  \vep_a$; \label{df8a}	
    \item $( y_a-y_{a+1}) \vep_a = - \vep_a$. \label{df8b}	
\end{enumerate}
\end{enumerate}
\end{df}
Note that the relations without the appearance of any $y_j$ are precisely the defining relations of the {\em periplectic Brauer algebra} $A_d$ studied in \cite{Co, KT, Mo}. On the other hand, these relations appeared in \cite[Definition 2.1]{ES}, \cite{Na} by setting all $w_k=0$ there, except for few differences in signs.

One may actually exclude $(P.4)$ from the defining relations as indicated in the following lemma. We keep it here so that one can compare with the relations in \cite[Definition 2.1]{ES}.
\begin{lemma} \label{relation4}
The defining relation $(P.4)$ can be derived from relations $(P.3)$, $(P.5)$ and $(P.8)$.
\end{lemma}
\begin{proof}
By relations $(P.8)$(a), $(P.5)$(c) and $(P.3)$, we have $$\vep_1 y_2^2 \vep_1 = \vep_1 y_2(y_2 \vep_1) =  \vep_1 y_2(y_1 \vep_1+\vep_1) =  \vep_1 y_1(y_1 \vep_1+\vep_1)+ \vep_1(y_1 \vep_1+\vep_1) = \vep_1y_1^2\vep_1 +2\vep_1y_1\vep_1.$$
 On the other hand,  by $(P.8)$(b), $(P.5)$(c) and $(P.3)$, we have $$\vep_1 y_2^2 \vep_1 = (\vep_1 y_2)y_2 \vep_1 =  (\vep_1y_1-\vep_1)y_2\vep_1 = (\vep_1y_1-\vep_1)y_1\vep_1- (\vep_1y_1-\vep_1)\vep_1 = \vep_1y_1^2\vep_1 -2\vep_1y_1\vep_1.$$ Therefore, $\vep_1y_1\vep_1 =0$.
The general case follows from a similar argument and induction on $k$.
\end{proof}

\subsection{An Arakawa-Suzuki type functor $\mc F$}
Let $M$ be an arbitrary $\frak{g}$-module. Denote by $C_{M, V} \in \text{End}_{\frak{g}}(M\otimes V)$ the corresponding image of the Casimir-like element $C$  in the action $C\curvearrowright M\otimes V$.
Denote by $s \in \text{End}_{\G}(V^{\otimes 2})$ the super-permutation
  \begin{align}
  s(e_a\otimes e_b) = (-1)^{\overline{e_a}\cdot \overline{e_b}}e_b\otimes e_a, \ \ \text{ for }a,b\in I.
  \end{align}
Furthermore, define $\vep \in \text{End}_{\G}(V^{\otimes 2})$ by \begin{align} \label{def::vep}\vep=2C_{V, V} -s.\end{align}

  \begin{lemma} \label{FormulaForelemente}
  \begin{align*}
  \vep(e_a\otimes e_b) = (e_a,e_b) \sum_{i=1}^n (e_i\otimes e_{\overline{i}} - e_{\overline{i}} \otimes e_i)=  \left\{ \begin{array}{ll} 0,  \text{ if } a \not = \overline{b},
\\  \sum_{i\in I} (-1)^{\overline{e_{i}}} (e_i\otimes e_{\overline{i}}), \text{ if }  a = \overline{b}.
 \end{array} \right.
  \end{align*}
 \end{lemma}

By abusing notation, we define the following elements in $\text{End} (M\otimes V^{\otimes d})$:
\begin{align}\label{VWact}
s_a = \text{Id}^{\otimes a} \otimes s \otimes \text{Id}^{\otimes {d-a-1}}, \ \ \vep_a= \text{Id}^{\otimes a} \otimes \vep \otimes \text{Id}^{\otimes {d-a-1}}, \ \  y_j = 2C_{M\otimes V^{\otimes {j-1}}, V} \otimes \text{Id}^{\otimes {d-j}},
\end{align} for $1\leq a \leq d-1$, $1\leq j\leq d$. It shall be clear from the context when we regard these elements as elements in $\Pd$ and when as elements in $\text{End} (M\otimes V^{\otimes d})$.

Moreover,
for $1\leq i<j\leq d$ and homogenous $m\in M$
, we define $\Omega_{i,j}$ and $\Omega_{0,j}$ in  $\End (M\otimes V^{\otimes d}) $ by
\begin{align*} \Omega_{i,j}&(m\otimes e_{t_1}\otimes e_{t_2}\otimes \cdots \otimes e_{t_d}) \\
&:=2\sum_{k=1}^{\dim\frak{g}}(-1)^{\ov{x_k}(\sum_{s=1}^{i-1}\ov{e_{t_s}}+\ov{m})+\ov{x_k}(\sum_{s=1}^{j-1}\ov{e_{t_s}}+\ov{m})}(m\otimes e_{t_1}\otimes  \cdots  \otimes x_{k}e_{t_i} \otimes  \cdots \otimes x_k^*e_{t_j}\otimes  \cdots  \otimes e_{t_d}),\\
\Omega_{0,j}&(m\otimes e_{t_1}\otimes e_{t_2}\otimes \cdots \otimes e_{t_d}) \\
&:=2\sum_{k=1}^{\dim\frak{g}}(-1)^{\ov{x_k}(\sum_{s=1}^{j-1}\ov{e_{t_s}})}(x_km\otimes e_{t_1}\otimes e_{t_2}\otimes  \cdots  \otimes x_{k}^*e_{t_j} \otimes \cdots \otimes e_{t_d}). \end{align*}
We may observe that \begin{align}\label{def::PolynomialVars}y_j=\sum_{0\leq k\leq j-1}\Omega_{k,j},\end{align} for all $1\leq j\leq d.$

 \begin{lemma} \label{BrauerAction}
  The elements $s_a$, $\vep_a$, $y_j$ defined by (\ref{VWact}) commute with the natural $\mf g$-action on $M\otimes V^{\otimes d}$. In other words, they belong to $\End_{\mf g} (M\otimes V^{\otimes d})$.
  \end{lemma}
 \begin{proof}
 For $\vep_a$, the statement follows from Lemma~\ref{FormulaForelemente}.
 For $s_a$, a similar operator is defined in \cite[Section 5.1]{CW}, which is known to commute with the action of $\gl(n|n)$ and hence commutes with the action of $\frak{g}$.
 For $y_j$, replacing $M$ by $M\otimes V^{\otimes j-1}$ in Proposition \ref{Prop::SC}, and the lemma follows.
  \end{proof}

\begin{prop}\label{not4}
The elements $\{s_a, \vep_a, y_j \,|\, 1\leq a\leq d-1,\,\, 1\leq j\leq d\}$ in $\End_\frak{g}(M\otimes V^{\otimes d})$ satisfy the relations (P.1)--(P.3) and (P.5)--(P.8).
\end{prop}
\begin{proof}
We check by definition that these relations hold in $\End_\frak{g}(M\otimes V^{\otimes d})$.
Relations $(P.1)$, $(P.2)$, $(P.5)$(a), $(P.5)$(c), $(P.6)$ are straightforward to check, where some of them can be found in \cite{Co, Mo}; $(P.3)$ is an immediate consequence of Lemma~\ref{FormulaForelemente}; $(P.8)$(a) and $(P.8)$(b) are somehow involved and they are verified in Appendixes \ref{PfOf8a} and \ref{PfOf8b}.

Consider $(P.5)$(b). The proof is similar to \cite[Lemma 8.4]{ES}. If $i\leq a+1$ then it obviously holds. Assume that $i> a+1$.
Since the parity of the basis element $x_j\in \frak{g}$ and its dual basis element $x_j^*\in\frak{g}^*$ must be the same, we may conclude that
\begin{center} $\Omega_{k,i}\circ \vep_a=\vep_a\circ \Omega_{k,i},$
for all $k\neq a, a+1$.\end{center}
Consequently, $(P.5)$(b) is equivalent to
$$(\Omega_{a,i}+\Omega_{a+1,i})\circ \vep_a=\vep_a\circ (\Omega_{a,i}+\Omega_{a+1,i}).$$
The last equality follows from Lemma \ref{lem::ComLem1} and $(P.5)$(b) is verified.

We now check $(P.7)$. Let $m\in M\otimes V^{\otimes a-1}$ and $v\in V^{\otimes d-a-1}$ be homogenous elements. Then for all $i,j \in I$, we have
  \begin{align*}
  &(y_{a+1}\left(s_{a}(m\otimes e_i\otimes e_j\otimes v)\right)) = y_{a+1}\left((-1)^{\ov{e_i}\ov{e_j}}m\otimes e_j \otimes e_i\otimes v\right)\\
  & =2\sum_{k}\left((-1)^{\ov{e_i}\ov{e_j}+(\ov{m}+\ov{e_j})\ov{x_k}}x_k(m\otimes e_j) \otimes x_k^*e_i\otimes v\right)  \\
  & =2\sum_{k}\left((-1)^{\ov{e_i}\ov{e_j}+(\ov{m}+\ov{e_j}) \ov{x_k} }x_km\otimes e_j \otimes x_k^*e_i\otimes v\right) + 2 \sum_{k}\left((-1)^{\ov{e_i}\ov{e_j}+\ov{e_j}\ov{x_k}}m\otimes x_ke_j \otimes x_k^*e_i\otimes v\right) \\
  &=  s_ay_a(m\otimes e_i\otimes e_j\otimes v)+ 2\sum_{k}((-1)^{\ov{e_i}\ov{e_j}+\ov{e_j}\ov{x_k}}m\otimes x_ke_j \otimes x_k^*e_i\otimes v) \\
    &=  s_ay_a(m\otimes e_i\otimes e_j\otimes v)+ m\otimes \left((-1)^{\ov{e_i}\ov{e_j}}2C_{V,V}(e_j\otimes e_i)\right)\otimes v. \\
    &=  s_ay_a(m\otimes e_i\otimes e_j\otimes v)+ m\otimes \left((s+\vep)\circ s(e_i\otimes e_j)\right)\otimes v.\\
    &=  (s_ay_a+\text{Id}+\vep_a)(m\otimes e_i\otimes e_j\otimes v),
  \end{align*} where the last equality comes from $(P.6)$(a).
  The other equality of $(P.7)$ can be proved in a similar method.
\end{proof}


As a consequence, the following $\mf p(n)$ analogue of the Arakawa-Suzuki functor (cf.~\cite{AS}) is established.
 \begin{thm} \label{thm:first} Let $M$ be any $\mf p(n)$-module.
Then we have the following right action of $\Pd$ on $M\otimes V^{\otimes d}$:
 \begin{align} v\cdot  s_a:=s_a(v),  \quad  v\cdot  \vep_a:=\vep_a(v), \quad   v\cdot  y_j:=y_j(v),\label{AffinePnAction} \end{align}
for all $v\in M\otimes V^{\otimes d}$, $1\leq a\leq d-1$, $1\leq j\leq d$.
In other words, there is an algebra homomorphism
\begin{align} \label{Thm::ASThm}
\Psi_M:\Pd  \longrightarrow \End_{\mf g}(M \otimes V^{\otimes d})^{\emph{opp}}.
\end{align}
 \end{thm}
\begin{proof}
By Lemma~\ref{BrauerAction}, Proposition~\ref{not4}, and Lemma~\ref{relation4}, all defining relations of $\Pd$ are preserved under $\Psi_M$.
\end{proof}

Let $\Pd$-mod denote the category of $\Pd$-modules. Observe that Theorem \ref{thm:first} defines a functor 
$$\mc F(\bullet):=\bullet \otimes V^{\otimes d}$$ 
from the category of $\frak{g}$-modules to $\Pd$-mod. The properties of $\mc F$ are studied in \cite{B+9}.

\section{Tensor product representations}\label{Section::Tensors}
In this section, we evaluate the functor $\mc F$ in Theorem \ref{thm:first} at various $\frak{g}$-modules $M=V^{\otimes k}$ for $k\in\Z_+$ to obtain a connection between tensor product representations of $\mf p(n)$ and $\Pd$.

We start with the simplest case where $M=\C$, the trivial representation. It turns out that the image $\Psi_\C(\Pd)$ is in fact the periplectic Brauer algebra in \cite{Co, Mo}.

Firstly we recall the Brauer $(d,d)$-diagram realization of $A_d$ in \cite{Co, KT, Mo}. A Brauer $(d,d)$-diagram is a graph with two rows of $d$ vertices $\{1,2,\ldots, d\}$ and $\{\ov 1, \ov 2, \ldots, \ov d\}$, $i$ above $\ov i$ for all $1\leq i \leq d$, and $d$-edges such that each vertex is connected to precisely one edge. That is, Brauer $(d,d)$-diagrams correspond to all partitions of $2d$-dots into pairs. By definition, the set $\mc G(d)$ of all Brauer $(d,d)$-diagrams forms a basis of Brauer algebras in \cite{Br}.

 Let $\mc A$ be the periplectic Brauer category in which objects are positive integers and morphisms are $(d,d)$-Brauer diagrams, see, e.g., \cite[Section 2.1]{Co}. Then $$A_d = \text{End}_{\mc A}(d).$$  We recall the generators $\widehat{s}_i$ and $\widehat{\vep}_i$ for $A_d$. For each $1\leq i \leq d-1$, denote by $\widehat{s}_i:= (i, i+1)\in A_d$ the Brauer diagram with a line connecting the upper vertex $i$ to the lower vertex $\ov{i+1}$, and with a line connecting the upper vertex $i+1$ to the lower vertex $\ov{i}$, and a line connecting $j$ to $\ov{j}$ for each $j\neq i,i+1$. Also, let $\widehat{\vep}_i:=\ov{(i,i+1)} \in A_d$ correspond to the Brauer diagram which consists only non-crossing propagating lines except for one cup and cap, connecting $\{i,i+1\}$ and $\{\ov i, \ov {i+1}\}$, respectively. The other elements $(i,j)$ and $\ov{(i,j)}$ in $A_d$ are defined analogously (see, e.g., \cite[Section 2.1.5]{Co}).
For example,

$$
\begin{tikzpicture}[scale=0.7,thick,>=angle 90]
\begin{scope}[xshift=4cm]
\draw (-6,0.5) node[] {$\widehat{\varepsilon}_i:=$};
\draw  (-5,0) -- +(0,1);
\draw [dotted] (-4.65,.5) -- +(0.4,0);
\draw  (-4,0) -- +(0,1);
\draw (-3.2,0) to [out=90, in=180] +(0.3,0.3);
\draw (-2.6,0) to [out=90, in=0] +(-0.3,0.3);
\draw (-3.2,1) to [out=-90, in=180] +(0.3,-0.3);
\draw (-2.6,1) to [out=-90, in=0] +(-0.3,-0.3);
\draw  (-1.8,0) -- +(0,1);
\draw [dotted] (-1.5,.5) -- +(0.4,0);
\draw  (-0.85,0) -- +(0,1);
\put(12,22){\tiny{$i$}}\put(22,22){\tiny{$i+1$}}
\put(12,-10){\tiny{$\overline{i}$}}\put(22,-10){\tiny{$\overline{i+1}$}}

\draw (2,0.5) node[] {$\widehat{s}_i:=$};
\draw  (3,0) -- +(0,1);
\draw (4.6,0) to [out=65,in=-115] +(0.6,1);
\draw (5.2,0) to [out=115,in=-65] +(-0.6,1);
\draw [dotted] (3.35,.5) -- +(0.4,0);
\draw  (4,0) -- +(0,1);
\draw  (6,0) -- +(0,1);
\draw [dotted] (6.3,.5) -- +(0.4,0);
\draw  (7,0) -- +(0,1);
\put(167,22){\tiny{$i$}}\put(178,22){\tiny{$i+1$}}
\put(167,-10){\tiny{$\overline{i}$}}\put(178,-10){\tiny{$\overline{i+1}$}}
\end{scope}
\end{tikzpicture}
$$

 In \cite{Mo}, Moon proved that there is a tensor product representation for $A_d$ acting as a subalgebra of centralizer:
\begin{align} \label{Moonsmapping}
\psi:A_d \rightarrow \text{End}_{\mf p(n)}(V^{\otimes d }).
\end{align}
The homomorphism $\psi$  in \eqref{Moonsmapping} is surjective for all $n,d\in \N$ (see, e.g., \cite[Lemma 8.3.3]{Co}). Furthermore, $\psi$ is an isomorphism if $n \geq d$  \cite[Theorem~4.5]{Mo}.

Set $M=\C$, the trivial representation, in Theorem~\ref{thm:first} and we have the tensor product representation of $\Pd$ on $\C \otimes V^{\otimes d} \cong V^{\otimes d}$.
In particular, by (\ref{VWact}), $\Psi_\C(s_a), \Psi_\C(\vep_a)$ coincide with the image of the generators $\widehat{s}_a$, $\widehat{\vep}_a$ in \eqref{Moonsmapping}.
That is, we have  $\Psi_\C(s_a) = \psi(\widehat{s}_a), \Psi_\C(\vep_a) = \psi(\widehat{\vep}_a)$.

Next we consider the image $\Psi_\C(y_j)$. By (\ref{VWact}) again, $\Psi_\C(\Omega_{0,j})=0$ for all $j$. In particular, $\Psi_\C(y_1)$ is the zero operator.
Using the defining relations in $A_d$ \cite[Proposition 2.1]{Mo}, we have the following identities for any $1\leq i<j\leq d$:
$$\psi(\,\,\ov{(i,j)}\,\,) = \psi((i,i+1))\cdots \psi((j-2,j-1))\cdot\psi(\,\,\ov{(j-1,j)}\,\,)\cdot\psi((j-2,j-1))\cdots \psi((i,i+1)).$$
This implies that $\Psi_\C(\Omega_{ij})=\psi((i,j))+\psi(\ov{(i,j)})$, for all $1\leq i<j\leq d$.
By (\ref{def::PolynomialVars}), we have
$$
\Psi_\C(y_j)=\sum_{0\leq k\leq j-1}\Psi_\C(\Omega_{k,j})=\sum _{1\leq k\leq j-1}\psi((k,j))+\psi(\,\,\ov{(k,j)}\,\,).
$$

Recall that the Jucys-Murphy elements $\{z_j \,|\, 1\leq j\leq d\}$ of $A_d$ are defined in \cite[Section 6.1.1]{Co} by setting $z_1=0$ and
\begin{align} \label{Eq::JM-element} &z_j:= \sum_{1\leq k\leq j-1}(k,j)+\ov{(k , j)}.\end{align}
Our discussion above shows that $\Psi_{\C}(y_j) = \psi(z_j)$ for all $1\leq j\leq d$.

In fact, $A_d$ is a homomorphic image of $\Pd$, where the Jucys-Murpy elements $z_j\in A_d$ are precisely the image of the polynomial generators $y_j\in \Pd$. This is parallel to the same phenomenon appeared in the degenerate affine Hecke algebra and the group algebra of the symmetric group. A similar phenomenon also appeared in the degenerate affine Nazarov-Wenzel algebra and the Brauer algebra, see \cite[Section~4]{Na}.

\begin{thm} \label{Thm::Quotient}
The map $\pi:\Pd\longrightarrow A_d$ given by \begin{align} \pi(s_a)=\widehat{s}_a, ~\pi(\vep_a)=\widehat{\vep}_a,~ \pi(y_j)=z_j, \end{align} is an algebra epimorphism.
\end{thm}
\begin{proof}
We check that the map $\pi$ preserves the defining relations $(P.1)$--$(P.8)$.
It suffices to check only the relations $(P.2)$(c), $(P.4)$, $(P.5)$(b), $(P.5)$(c), $(P.7)$, $(P.8)$(a), $(P.8)$(b). Relations $(P.2)$(c) and $(P.5)$(b) follow from \cite[Lemma 6.3.3]{Co}; $(P.5)$(c) follows from \cite[Lemma 6.1.2]{Co}; $(P.4)$ follows from \cite[Corollary 6.3.4]{Co}; $(P.7)$ follows from \cite[Lemma 6.3.1]{Co} together with $(P.6)$(a); $P.8$(a) and $(P.8)(b)$ follow from \cite[Lemma 6.3.1]{Co}.
\end{proof}

\begin{rem}
 When $n\geq d$, the homomorphism \eqref{Moonsmapping} is an isomorphism \cite[Theorem 4.5]{Mo}. In this situation, the map $\pi=(\psi)^{-1}\circ \Psi_\C$ gives a simple proof of the above theorem.
\end{rem}

Next we are concerned with more general tensor product representations of $\Pd$. Namely, we consider $\Psi_{M}$ with $M = V^{\otimes m}$, $m \in \Z_{+}$. The following is a $\mf{p}(n)$-analogue of results in \cite[Section~4]{Na}. In particular, the map $\pi_0$ is exactly the map $\pi$ in Theorem \ref{Thm::Quotient}.
\begin{thm}\label{shiftpi}
  For each $m \in \Z_+$, there is an algebra homomorphism $\pi_m:\Pd\longrightarrow A_{m+d}$ such that
  \begin{align}
  \pi_m(s_a)= \widehat{s}_{m+a}, ~\pi_m(\vep_a)=\widehat{\vep}_{m+a}, ~ \pi_m(y_j)=z_{m+j},
    \end{align} for each $1\leq a \leq d-1$ and $1\leq j\leq d$.
\end{thm}
\begin{proof} In $A_d$, we have $\widehat{\vep_i} \, z_i^k \, \widehat{\vep_i} =0$ for all $1\leq i \leq m+d$ and $k\in \N$ by \cite[Corollary~6.3.4]{Co}. It implies that defining relations of $\Pd$ are all preserved under $\pi_m$ for any $m\in\Z_+$.
\end{proof}

 \vskip 1cm

\section{A PBW basis theorem for $\Pd$} \label{Section::PBWbasis}
In this section, we give a PBW type basis for $\Pd$ by adapting the approach in \cite[Theorem~4.6]{Na}.

\subsection{Regular monomials}
We first recall the notion of regular monomials defined by Nazarov in \cite[(4.18)]{Na}.
Recall that $A_d$ is a diagram algebra with basis $\mc G(d)$ of Brauer $(d,d)$-diagrams.
In the rest of this article, adapting the notation in $A_d$, we set
$(i,i+1):=s_i$, $(\ov{i,i+1}):=\vep_i$ to be the generators of $\Pd$, while
the general elements $(i,j)$ and $\ov{(i,j)}$ for $1\leq i<j\leq d$ in $\Pd$ are defined analogously.

Recall that any edge of a graph in $\mc G(d)$ connecting the vertices $\{i,j\}$ or connecting $\{\ov i, \ov j\}$ for some $i<j$ will be called a {\em horizontal edge}, while the vertex $j$ or $\ov{j}$ will be called a {\em right end}. A monomial $u\in \Pd$ in $s_a, \vep_b, y_i$ is called {\em regular} (see also Definition \ref{Der::RegularDotBrauerDiagram}) if  \begin{align}\label{def::RegularMonomials} u= y_1^{i_1}y_2^{i_2}\cdots y_d^{i_d} \cdot \gamma \cdot y_1^{j_1}y_2^{j_2}\cdots y_d^{j_d},\end{align} where  $\gamma \in \mc G(d)$ satisfying the following conditions:
\begin{align}
\text{if } &k\in \{r_1, \ldots r_q\} \text{ then } i_k =0, \label{1stRegularity}\\
\text{if } &j_t \neq 0 \text{ then } t\in \{\overline{r'_1}, \ldots \overline{r'_q}\},\label{2ndRegularity}
\end{align} where $r_1,r_2,\ldots ,r_q\in \{1,2,\ldots ,d\}$ (resp. $\overline{r'_1}, \ldots \overline{r'_q} \in \{\ov 1,\ov 2, \ldots ,\ov d\}$) are all upper (resp. lower) right ends of horizontal edges of $\gamma$. 
The following theorem is  the main result in this section.

\begin{thm}\label{thm::PBW}
  The set of all regular monomials forms a linear basis for $\Pd$.
\end{thm}

The rest of this section is devoted to the proof of Theorem \ref{thm::PBW}.
We first equip $\Pd$ with a filtration by setting degrees of the generators as follows:
\begin{align*}
\text{deg}(s_a) = \text{deg}(\vep_a) =0,~ \text{deg}(y_j) =1.
\end{align*}
Denote by $w_a$ the image of $y_a$ in the corresponding graded algebra $\text{gr} \Pd$.  The following  identity in $\text{gr} \Pd$ comes immediately from $(P.2)$(c) and $(P.7)$:
\begin{align} \label{FirEq}
\tau w_a\tau^{-1} = w_{\tau(a)}, \quad \tau\in \mf S_d.
\end{align}
By \eqref{FirEq} and $(P.4)$, it follows that \begin{align} \overline{{(i,j)}} w_i^k \overline{{(i,j)}} =0,  \label{2ndEq} \end{align} for all $k\in \N$ and $1\leq i\neq j\leq d$.
By \eqref{FirEq}, $(P.2)$(c), $(P.5)$(b), $(P.7)$, $(P.8)$(a) and $(P.8)$(b), we obtain the following identities:
\begin{align}
&\overline{{(i,j)}}w_a =w_a\overline{{(i,j)}}, \text{ for $a \neq i, j$},\label{3rdEq} \\
&\overline{{(i,j)}}(w_i - w_j) = 0 =  (w_i - w_j)\overline{{(i,j)}},\text{ for $i \neq j$}. \label{4thEq}
\end{align}
The  identities \eqref{FirEq}, \eqref{2ndEq}, \eqref{3rdEq} and \eqref{4thEq} provide  all ingredients for the following two lemmas.
\begin{lemma}\label{NazarovLemma44}\cite[Lemma 4.4]{Na}
  Let $w$ be a monomial in $w_1,w_2,\ldots, w_d$. Let $\gamma,\gamma' \in \mc G(d)$, then we have the following equality in $\emph{gr}\Pd$:
$$\gamma w  \gamma' =\varepsilon w' \gamma \gamma' w'',$$ where $w',w''$ are monomials in $w_1,w_2,\ldots ,w_d$ and $\varepsilon\in\{0,1\}$.
\end{lemma}

\begin{proof}
Let $2r$ and $2s$ be the number of horizontal edges of $\gamma$ and $\gamma^\prime$, respectively. We prove the statement by induction on $\min\{r,s\}$ and on the degree of $w$. The case when $r=s=0$ or $\text{deg}~w=1$ are trivial. Assume that $r,s\geq 1$ and $\text{deg}~w\geq 1$. We explain the case when $s\leq r$ explicitly here, while the case can be deduced in a very similar way.

By defining relations, we may suppose that $\gamma=\overline{{(k_1,\ell_1)}} \cdots \overline{{(k_r,\ell_r)}} \tau_1$ and 
$\gamma^\prime=\overline{{(k_1^\prime,\ell_1^\prime)}} \cdots \overline{{(k_s^\prime,\ell_s^\prime)}} \tau_2$ 
for some $\tau_1,\tau_2\in \mf S_d$, where $k_1,\ell_1,\ldots,k_r,\ell_r$ are pairwise distinct and so are $k_1^\prime,\ell_1^\prime,\ldots,k_s^\prime,\ell_s^\prime$.  
By equation (\ref{FirEq}), we may further assume that $\tau_1=\tau_2=1$.

Consider the monomial $w=w_1^{i_1}\, \cdots \, w_n^{i_n}.$   
Choose any index $1\leq k\leq n$ such that $i_k\neq 0$. 
If $k\notin \{ k_1,\ell_1,\ldots, k_r,\ell_r\}$, then $\gamma w_k^{i_k}= w_k^{i_k}\gamma$ by (\ref{3rdEq}), and we have done by induction on the degree of $w$. Similarly, if $k\notin \{ k_1^\prime ,\ell_1^\prime ,\ldots, k_s^\prime,\ell_s^\prime\}$, then $w_k^{i_k}\gamma^\prime= \gamma^\prime w_k^{i_k}$ and we have done.

Now we assume that $k=k_j=k_h^\prime$ for some $1\leq j\leq r$ and $1\leq h\leq s$. Let $\ell=\ell_j$ and let $\ell^\prime=\ell_h^\prime$.
Since $k_1,\ell_1,\ldots,k_r,\ell_r$ and $k_1^\prime,\ell_1^\prime,\ldots,k_s^\prime,\ell_s^\prime$ are pairwise distinct, we may suppose that 
$$\gamma=\hat\gamma \,\, \overline{{(k,\ell)}} \quad \text{and} \quad 
\gamma^\prime=\overline{{(k,\ell^\prime)}} \,\, \hat\gamma^\prime,$$ 
where $\hat\gamma$ and $\hat\gamma^\prime$ are the graphs obtained by removing the factor $\overline{{(k,\ell)}}$ in $\gamma$ and removing the factor $\overline{{(k,\ell^\prime)}}$ in $\gamma^\prime$, respectively.

Suppose that $\ell=\ell^\prime$. By $(\ref{2ndEq})$, $(\ref{3rdEq})$ and $(\ref{4thEq})$, we have
\begin{align*}
\gamma w \gamma^\prime & =   \hat\gamma \,\, \overline{{(k,\ell)}} \,\, w_k^{i_k} w_\ell^{i_\ell} \prod_{m\neq k,\ell} w_m^{i_m} \,\,\overline{{(k,\ell)}} \,\,\hat\gamma^\prime\\
 &= \hat\gamma \,\, \overline{{(k,\ell)}} \,\, w_\ell^{i_k+i_\ell} \,\, \overline{{(k,\ell)}} \,\, \prod_{m\neq k,\ell} w_m^{i_m} \,\, \hat\gamma^\prime=0.
\end{align*}
Finally, suppose that $\ell\neq\ell^\prime$. By a similar argument, we have
\begin{align*}
\gamma w \gamma^\prime & =   \hat\gamma \,\, \overline{{(k,\ell)}} \,\, w_k^{i_k} w_\ell^{i_\ell} u_{\ell^\prime}^{i_{\ell^\prime}} \,\, \prod_{m\neq k,\ell, \ell^\prime} w_m^{i_m} \,\, \overline{{(k,\ell^\prime)}} \,\, \hat\gamma^\prime\\
& =   \hat\gamma \,\, \overline{{(k,\ell)}} \,\, w_k^{i_k+i_{\ell^\prime}} w_\ell^{i_\ell} \,\, \overline{{(k,\ell^\prime)}} \,\, \prod_{m\neq k,\ell,\ell^\prime}  w_m^{i_m}  \,\, \hat\gamma^\prime\\
 &= \hat\gamma \,\, \overline{{(k,\ell)}} \,\, w_\ell^{i_k+i_\ell+i_{\ell^\prime}} \,\, \overline{{(k,\ell^\prime)}} \,\, \prod_{m\neq k,\ell,\ell^\prime} w_m^{i_m}  \,\, \hat\gamma^\prime\\
  &= \hat\gamma \,\, \overline{{(k,\ell)}} \,\, \overline{{(k,\ell^\prime)}} \,\, w_\ell^{i_k+i_\ell+i_{\ell^\prime}} \,\,\prod_{m\neq k,\ell,\ell^\prime} w_m^{i_m} \,\, \hat\gamma^\prime.
  \end{align*}
Note that the number of horizontal edges in $\hat\gamma^\prime$ is less than $2s$ and the induction applies.
\end{proof}

\begin{lemma}\label{NazarovLemma45}\cite[Lemma 4.5]{Na}
Let $w,w'$ be monomials in $w_1,w_2,\ldots, w_d$. Then we have the following equality in $\emph{gr}\Pd$:
  $$w\gamma w' =  w_1^{i_1} \cdots w_d^{i_d}\gamma w_1^{j_1}\cdots w_{d}^{j_d},$$ where the exponents $i_1, \ldots, i_d$ and $j_1,\ldots ,j_d$ satisfy \eqref{1stRegularity} and \eqref{2ndRegularity}.
\end{lemma}
\begin{proof}
	Similar to the proof of Lemma \ref{NazarovLemma44}, we may assume that 
$\gamma$ has exactly $2r$ horizontal edges and 
$$\gamma = \ov{(k_1,\ell_1)} \cdots \ov{(k_r,\ell_r)} \cdot \sigma$$ 
for some $\sigma\in \mathfrak{S}_d$ where $k_1,\ell_1,\ldots,k_r,\ell_r$ are pairwise distinct. 
	Suppose that $w = w_1^{p_1}\cdots w_{d}^{p_d}$ and $w' = w_1^{q_1}\cdots w_{d}^{q_d}$. By \eqref{FirEq}, we have the following equality in $\gr\Pd$
\begin{equation}\label{L5.3}
w\gamma w' =  w_1^{p_1}\cdots w_{d}^{p_d} \cdot \ov{(k_1,\ell_1)} \cdots \ov{(k_r,\ell_r)}  \cdot w_{\sigma(1)}^{q_1}\cdots w_{\sigma(d)}^{q_d}\cdot \sigma.
\end{equation}
Suppose that $p_m\neq 0$ for some $1\leq m\leq d$ which is an upper right end of some horizontal edge of the graph $\ov{(k_1,\ell_1)} \cdots \ov{(k_r,\ell_r)}$. We may assume $\ov{(k_1,\ell_1)}=\ov{(n,m)}$ where $n$ is the upper left end connecting $m$.

By \eqref{4thEq}, we have
$$ \prod_{i=1}^d w_i^{p_i}  \cdot \prod_{j=1}^r \ov{(k_j,\ell_j)} 
=
 (\prod_{i\neq m}^d w_i^{p_i} ) \cdot w_m^{p_m} \cdot \ov{(m,n)} \cdot \prod_{j\neq1}^r \ov{(k_j,\ell_j)}=
 (\prod_{i\neq m}^d w_i^{p_i})  \cdot w_n^{p_m} \cdot \ov{(m,n)} \cdot \prod_{j\neq1}^r \ov{(k_j,\ell_j)} 
$$
Repeat this process, we show that one can always rewrite \eqref{L5.3} such that the exponents $p_1,\ldots,p_d$ satisfy the condition \eqref{1stRegularity}. The argument for \eqref{2ndRegularity} is similar.

\end{proof}

The following result follows immediately from Lemma \ref{NazarovLemma44} and Lemma \ref{NazarovLemma45}.
\begin{cor}\label{NazarovLemma46}
 $\Pd$ is spanned by the set of all regular monomials.
\end{cor}

To prove Theorem \ref{thm::PBW}, it remains to prove the linear independence of regular monomials. By induction on degree, it suffices to prove the linear independence of those terms having maximal degree $\sum_{q=1}^{d}i_q+j_q$ ($i_q$, $j_q$ given as in \eqref{def::RegularMonomials}) in a linear combination of regular monomials.
This can be obtained by a similar method as employed in the proof of \cite[Lemma 4.8]{Na} by dots tracing from the upper row and the lower row. 
\begin{proof}[Proof of Theorem \ref{thm::PBW}] 
Let $\{F_1,\ldots , F_s\}$ be a set of regular monomials in $\Pd$.  Let $F=\sum_{k=1}^{s} \alpha_kF_k$, where $\alpha_k\in \C\backslash\{0\}$. Our goal is to show that $F\neq 0$.
For each $1\leq k \leq s$, we choose a regular monomial expression of $F_k$, say \begin{align}\label{Eq::ChoiceOfRegularEx} &F_k = y_1^{i^{(k)}_1}y_2^{i^{(k)}_2}\cdots y_d^{i^{(k)}_d} \cdot \gamma^{(k)} \cdot y_1^{j^{(k)}_1}y_2^{j^{(k)}_2}\cdots y_d^{j^{(k)}_d}. \end{align}  Then we define the degree of $F_k$ to be the number $d(F_k): = \sum_{q=1}^{d}i^{(k)}_q+j^{(k)}_q$, which turns out to be proved to be independent of the choice of expression. Let $m$ be the maximal degree among $d(F_1), \ldots, d(F_s)$. For each $1\leq k \leq s$ we denote the number of horizontal edges of $\gamma^{(k)}$ by $2r^{(k)}$ and set $2r$ to be the minimal number among $2r^{(1)}, \ldots ,2r^{(s)}$. 

We proceed our argument by considering the tensor product representation $\pi_m$ in Theorem~\ref{shiftpi}. Consider the set $\mc G \subseteq \mc G(m+d)$ consisting of $(m+d,m+d)$-Brauer diagram $\Gamma$ satisfying the following three conditions:
\begin{itemize}
	\item[(i)] $\Gamma$ has exactly $2r$ horizontal edges.
	\item[(ii)]  There are no vertical edges in $\Gamma$ of the form $\{k, \ov k\}$ with $1\leq k \leq m$.
	\item[(iii)] There are no horizontal edges in $\Gamma$ of the form $\{k,\ell\}$ or $\{\ov k, \ov \ell\}$, with $1\leq k,\ell \leq m$.
\end{itemize} Let $\C\mc G$ be the subspace of $A_{m+d}$ spanned by $\mc G$.  Also, 
we define the projection $p_m$ of $A_{m+d}$ on $\C\mc G$. We may note that $p_m(\pi_m(F_k)) = 0$, if either  $d(F_k) < m$ or $2r^{(k)}\neq 2r$. To see this, if $2r^{(k)} \neq 2r$, then $p_m(\pi_m(F_k)) = 0$ by (i). If $d(F_k)<m$, then by \eqref{Eq::JM-element}, $\pi(F_k)$ must contain a vertical edge connecting $k$ and $\ov{k}$ for some $1\leq k \leq m$, and hence $p_m(\pi_m(F_k))= 0$. As a consequence, from now on we may assume that  $d(F_k) = m$ and $2r^{(k)} = 2r$, for each $1\leq k \leq s$. We shall show that $p_m(\pi_m(F_1)), \ldots, p_m(\pi_m(F_s))$ are linearly independent.


For each $1\leq k\leq s$, $p_m(\pi_m(F_k))$ is a linear combination of $(m+d,m+d)$-Brauer diagrams of the form: 
\begin{align} \label{Thm::Eq::BasicExpreesion} &\Pi_{\ell=1}^d c({N_{\ell,1}, m+\ell})\cdots c({N_{\ell,i_\ell}, m+\ell})\cdot \pi_m(\gamma^{(k)}) \cdot \Pi_{\ell=1}^d c({N_{\ell,1}', m+\ell}) \cdots c({N_{\ell,j_\ell}', m+\ell}), \end{align} where $c({a,b}) \in \{(a,b), \ov{(a,b)}\}$ for given $1\leq a\leq m,~ ,m+1 \leq b \leq m+d$, and there is a bijection between the following two sets  \begin{align}\label{Eq::BijectionCondition}&\bigcup_{\ell=1}^d\{N_{\ell,1}, N_{\ell,2},\ldots, N_{\ell,i_\ell},  ~N_{\ell,1}', N_{\ell,2}',\ldots, N_{\ell,j_\ell}'\}_{\ell = 1}^d \longleftrightarrow \{1,2,\ldots ,m\}.\end{align}

{\bf Claim 1}: In the expression \eqref{Thm::Eq::BasicExpreesion}, the element $p_m(\pi_m(F_k))$ is a linear combination of $(m+d,m+d)$-Brauer diagrams $$L_1\cdots L_d\cdot \pi_m(\gamma^{(k)})\cdot R_1 \cdots R_d,$$
where there are   $$\bigcup_{\ell=1}^d\{N_{\ell,1}, N_{\ell,2},\ldots, N_{\ell,i_\ell}, ~N_{\ell,1}', N_{\ell,2}',\ldots, N_{\ell,j_\ell}'\} =\{1,2,\ldots, m\},$$ satisfying condition \eqref{Eq::BijectionCondition} such that for each $1\leq \ell \leq d$ the products $L_\ell$ and $R_\ell$ are of the following forms 
\begin{align}&L_\ell = (N_{\ell,1}, m+\ell)\cdots(N_{\ell,i_\ell}, m+\ell), \text{or}~\ov{(N_{\ell,1}, m+\ell)}\cdots\ov{(N_{\ell,i_\ell}, m+\ell)}, \\   &R_{\ell}=(N_{\ell,1}', m+\ell)\cdots(N_{\ell,j_\ell}', m+\ell),\text{or}~ \ov{(N_{\ell,1}', m+\ell)}\cdots\ov{(N_{\ell,j_\ell}', m+\ell)}. 
\end{align} To prove this, suppose on the contrary that the coefficient of the element \begin{align} \label{Thm::Eq::G1}  &\cdots (a,b)\ov{(c,b)} \cdots \pi_m(\gamma^{(k)}) \cdots,\end{align} in the expression \eqref{Thm::Eq::BasicExpreesion} is non-zero. However, the Brauer diagram \eqref{Thm::Eq::G1} contains the horizontal edge $\{a,c\}$ with $1\leq a,c\leq m$, which is a contradiction. 
With exactly the same method, one can deduce that
the coefficients of the elements \begin{align*}   &\cdots \ov{(a,b)}(c,b) \cdots \pi_m(\gamma^{(k)}) \cdots, \\
   & \cdots\pi_m(\gamma^{(k)}) \cdots \ov{(a,b)}(c,b) \cdots, \\
  & \cdots\pi_m(\gamma^{(k)}) \cdots (a,b)\ov{(c,b)} \cdots,\end{align*} in the expression \eqref{Thm::Eq::BasicExpreesion} must be zero.

 Recall a leading term defined in \cite[Section 4]{Na} is a $(m+d,m+d)$-Brauer diagrams $T$ obtained in \eqref{Thm::Eq::BasicExpreesion} of the following form: \begin{align} \label{Thm::Eq::LeadingTerm}
 &T=\Pi_{\ell=1}^d(N_{\ell,1}, m+\ell)\cdots(N_{\ell,i_\ell}, m+\ell)\cdot \pi_m(\gamma^{(k)}) \cdot \Pi_{\ell=1}^d(N_{\ell,1}', m+\ell)\cdots(N_{\ell,j_\ell}', m+\ell),\end{align} where $$\bigcup_{\ell = 1}^d \{N_{\ell,1}, N_{\ell,2},\ldots, N_{\ell,i_\ell},~N_{\ell,1}', N_{\ell,2}',\ldots, N_{\ell,j_\ell}'\}=\{1,2,\ldots ,m\},$$ satisfying condition \eqref{Eq::BijectionCondition}. As used in the paragraph above, for each $1\leq \ell \leq d$, we set \begin{align} \label{Thm::Eq::LeadingTerm::LR} &L_\ell = (N_{\ell,1}, m+\ell)\cdots(N_{\ell,i_\ell}, m+\ell),~  R_{\ell}=(N_{\ell,1}', m+\ell)\cdots(N_{\ell,j_\ell}', m+\ell). 
 \end{align} A $(m+d,m+d)$-Brauer diagram which is obtained in \eqref{Thm::Eq::BasicExpreesion} and not a leading term is called non-leading term. \\
 
 {\bf Claim 2}: A leading term and a non-leading term can not be proportional.
 
 We suppose on the contrary that there are some $1\leq k'\leq s$, $1\leq q\leq d$ and a non-leading term \begin{align}&T'=L'_1\cdots L'_d  \cdot \pi_m(\gamma^{(k')}) \cdot R'_1\cdots R'_\ell.
  \end{align} with $$L'_q  = \ov{(M_{q,1}, m+q)}\cdots\ov{(M_{q,i'_q}, m+q)}\neq 1~(\text{i.e.}~ i'_q\geq 1)$$  such that
  $T'$ is proportional to $T$. 
  
  We first note that the mate of $m+q$ in the diagram $T'$ is $M_{q,1}$, namely, $T'$ has the horizontal edge $\{M_{q,1}, m+q\}$. If $L_{q} \neq 1$ then $T$ has the line $\{\ov{N_{q,1}},m+q\}$, this is a contradiction. Therefore we have $L_q = 1$. Now we trace the mate of $m+q$ in the Brauer diagram $T$. Since $\gamma^{(k)}$ is regular and $L_q =1$, we may conclude that the mate of $m+q$ in the Brauer diagram $T$ is identical to that in the Brauer diagram $$\pi_m(\gamma^{(k)}) \cdot \Pi_{\ell=1}^d(N_{\ell,1}', m+\ell)\cdots(N_{\ell,j_\ell}', m+\ell),$$ and so it must lie in $$\{m+1,\ldots m+d\}\cup \{\ov{m+1},\ldots \ov{m+d}\}.$$ This is a contradiction to $N_{q,1}\leq m$.
 
  Next we suppose on the contrary that there is $$R'_q  = \ov{(M'_{q,1}, m+q)}\cdots\ov{(M'_{q,j'_q}, m+q)}\neq 1 ~(\text{i.e.}~ j'_q \geq 1)$$  such that
 $T'$ is propositional to $T$. 
 
 Similarly, we note that $T'$ has the horizontal edge $\{\ov{M'_{q,j'_q}}, \ov{m+q}\}$. If $R_{q} \neq 1$ then $T$ has the line $\{{N_{q,j_q}},\ov{m+q}\}$, this is a contradiction. Since $\gamma^{(k)}$ is regular and $R_q =1$, we may conclude that the mate of $\ov{m+q}$ in the Brauer diagram $T$ is in $\{\ov{m+1},\ldots \ov{m+d}\}$. This is a contradiction to $M'_{q,j'_q}\leq m$. This completes the proof of Claim 2.\\
 
 {\bf Claim 3}: The Leading term $T$ in \eqref{Thm::Eq::LeadingTerm} is uniquely determined by the parameters $$(\gamma^k; i_1,\ldots ,i_d; j_1,\ldots, j_d; N_{1,1}, \ldots, N_{1,i_1},\ldots ;N'_{1,1},\ldots, N'_{1,j_1},\ldots). $$
 

For each $1\leq \ell \leq d$, we compute the parameters $i_\ell, j_\ell, N_{\ell, 1},\ldots , N_{\ell,i_\ell}, N'_{\ell, 1},\ldots , N'_{\ell,j_\ell}$ via the following two algorithms of tracing mate for a given number in $$\{1,\ldots, m+d,\ov 1, \ldots, \ov{m+d}\},$$ in the Brauer diagram $T$. \\ 

The following is the algorithm for computing parameters $i_\ell, N_{\ell, 1},\ldots , N_{\ell,i_\ell}$:\\

{\bf Step 0}: Set $N_{\ell,0}:=m+\ell$ and go to Step 1 with parameter $i=0$. 

{\bf Step 1 with parameter $i$}: \\
If the mate of $N_{\ell,i}$ in $T$ does not lie in $\{\ov 1, \ov 2,\ldots, \ov m\}$, then it must lie in $$\{m+1,\ldots, m+d,\ov{m+1},\ldots,\ov{m+d}\},$$ since $\gamma^{(k)}$ is regular. In this case we set $i_\ell$ to be $i$ and quit this algorithm. If the mate of $N_{\ell,i}$ in $T$ lie in $\{\ov 1, \ov 2,\ldots, \ov m\}$ then we define $N_{\ell,i+1}$ such that $\ov{N_{\ell,i+1}}$ is the mate of $N_{\ell,i}$  in $T$ and re-run Step 1 with parameter $i+1$.  \\

The following is the algorithm for computing parameters $j_\ell, N'_{\ell, 1},\ldots , N'_{\ell,j_\ell}$:\\

{\bf Step 0}{$^{\prime}$}: Set $N'_{\ell,0}:=m+\ell$ and go to Step $1^{\prime}$ with parameter $j=0$. 

{\bf Step 1}{$^{\prime}$ {\bf with parameter} $j$}: \\
If the mate of $\ov{N'_{\ell,j}}$ in $T$ does not lie in $\{1,2,\ldots, m\}$, then it must lie in $\{\ov{m+1},\ldots,\ov{m+d}\}$ since $\gamma^{(k)}$ is regular. In this case we set $j_\ell$ to be $j$ and quit this algorithm. If the mate of $\ov{N'_{\ell,j}}$ in $T$ lie in $\{1,2,\ldots, m\}$ then we define $N'_{\ell,j+1}$ to be the mate of $\ov{N'_{\ell,j}}$ in $T$ and re-run Step 1' with parameter $j+1$.  \\

These two algorithms determine desired parameters. As a consequence,  $\gamma^{(k)}$ is also obtained since elements $L_1,\ldots, L_d$ and $R_1,\ldots, R_d$ defined in \eqref{Thm::Eq::LeadingTerm::LR} are invertible.  This completes the proof of Claim 3.

Claim 1, Claim 2 and Claim 3 implies that $p_m (\pi_m(F_1)), p_m (\pi_m(F_2)), \ldots, p_m (\pi_m(F_s))$ are linearly independent. This completes the proof.
 
\end{proof}

Along with the proof of Theorem \ref{thm::PBW}, we obtain a $\mf p(n)$ counterpart of \cite[Theorem 4.7]{Na}.
 \begin{thm}
  We have
  $$\bigcap_{m\geq 0} \emph{Ker}(\pi_m) =0.$$
 \end{thm}

We conclude this section by a similar result in \cite[Corollary~4.9]{Na}.
Recall that the {\em degenerate affine Hecke algebra} $H_d$ (see e.g. \cite{Dr}) is generated by
the group algebra of the symmetric group $\C[\mf S_d]$ and the polynomial generators $v_1, v_2,\ldots,v_d$ subject
to the relations
\begin{enumerate}
\item $s_iv_j=v_js_i$    for $j \not\in \{i,i+1\}$ ,
\item $s_iv_i-v_{i+1}s_i=-1$, \qquad $s_iv_{i+1}-v_is_{i}=1$.
\end{enumerate}
 The following corollary can be observed from the defining relations of $\Pd$.
\begin{cor}
The map $s_a\mapsto s_a$, $\vep_a\mapsto 0$, $y_j\mapsto v_j$
defines a homomorphism from $\Pd$ to $H_d$.
\end{cor}

\section{A diagrammatic description for $\Pd$}\label{Section::diagram}


In this section, we extend the results in \cite{Mo, KT} so that a diagrammatic realization of $\Pd$ is obtained. 
Nevertheless, many results in Section \ref{Section::PBWbasis} can be translated into the language of diagrams, which provides an intuitive way to understand the meaning of the definitions and the arguments in the proofs there.

Note that the periplectic Brauer algebra $A_d$, which has a diagrammatic realization, is naturally contained in $\Pd$. We would like to extend the realization of $A_d$ to $\Pd$. It suffices to decide the diagrams of $y_j$ for $1\leq j\leq d$, and then to describe how the multiplication is performed involving the new defined diagrams.

Similar to the case in \cite{ES}, we define the graph of $y_j$ to be the graph obtained by adding a dot to the $j$-th vertical line of the unity in $A_d$:

$$
\begin{tikzpicture}[scale=0.7,thick,>=angle 90]
\begin{scope}[xshift=4cm]
\draw (2,0.5) node[] {${y}_j := $};
\draw  (3,0) -- +(0,1);
\draw (5,0) -- +(0,1);
\draw [dotted] (3.55,.5) -- +(0.4,0);
\draw  (4.4,0) -- +(0,1);
\draw  (5.6,0) -- +(0,1);
\draw [dotted] (6.1,.5) -- +(0.4,0);
\draw  (7,0) -- +(0,1);
\put(177,25){\tiny{$j$}}
\put(177,-10){\tiny{$\overline{j}$}}
\put(176.5, 13){$\bullet$}
\end{scope}
\end{tikzpicture}\\[3mm]
$$
Moreover, we allow a dot to freely move up or move down along a vertical line hence we will force them to be attached either on the top end or on the bottom end. For example, the following graphs are considered to be the same one but we use the first one or the last one to represent it:

$$\begin{tikzpicture}[scale=0.7,thick,>=angle 90]
\begin{scope}[xshift=4cm]
\draw (-4.2,0.5) node[] {${y}_1 = $};
\draw  (-3,0) -- +(0,1);
\draw  (-2,0) -- +(0,1);
\draw  (-1,0) -- +(0,1);
\put(17,13){$\bullet$};
\put (67.5,6){$\equiv$};
\draw  (0,0) -- +(0,1);
\draw  (1,0) -- +(0,1);
\draw  (2,0) -- +(0,1);
\put(77,7){$\bullet$};
\put (130,6){$\equiv$};
\draw  (3.2,0) -- +(0,1);
\draw  (4.2,0) -- +(0,1);
\draw  (5.2,0) -- +(0,1);
\put(141,0){$\bullet$};
\end{scope} \end{tikzpicture}$$

\begin{df}
Let $d\in\mathbb{N}$. A dot Brauer $d$-diagram is a Brauer $(d,d)$-diagram $\gamma$ with finitely many dots attached to the lines of $\gamma$ such that each dot is attached to the end of a straight line (could be top or bottom) or the end of a horizontal edge (could be left or right). 
\end{df}

\begin{df} \label{Der::RegularDotBrauerDiagram}
A dot Brauer $d$-diagram is called {\em regular} if its dots satisfy the following conditions:
\begin{enumerate}
\item If a dot appears in bottom, then it must be attached to the right end of a cap.
\item  A dot can not attach to the right end of a cup. 
\end{enumerate}
We denote the $\mathbb{C}$-span of all regular dot Brauer $d$-diagrams by $\hat{\mc G}(d)$
\end{df}

\begin{rem}
The restrictions for regular graphs are exactly the conditions for regular monomials given in (\ref{1stRegularity}) and (\ref{2ndRegularity}).

\end{rem}

For example, the first graph is a regular dot Brauer $5$-diagram, the second one is a dot Brauer $5$-diagrams but not a regular one, and the third one is not a dot Brauer $5$-diagram:
$$ \begin{tikzpicture}[scale=0.7,thick,>=angle 90]
\begin{scope}[xshift=4cm]
\draw  (-10.0,0) -- +(0,1);
\put(-122.5,14){$\bullet$}

\draw (-9.5,0) to [out=65,in=-115] +(0.6,1);
\put(-102,14){$\bullet$}

\draw (-8.9,0) to [out=115,in=-65] +(-0.6,1);

\draw (-8.2,0) to [out=90, in=180] +(0.3,0.3);
\draw (-7.6,0) to [out=90, in=0] +(-0.3,0.3);
\put(-73.8,0){$\bullet$}

\draw (-8.2,1) to [out=-90, in=180] +(0.3,-0.3);
\draw (-7.6,1) to [out=-90, in=0] +(-0.3,-0.3);

\put(-60,0){$,$}

\draw  (-4.5,0) -- +(0,1);
\draw (-4.1,0) to [out=65,in=-115] +(0.6,1);
\draw (-3.5,0) to [out=115,in=-65] +(-0.6,1);
\put(-4.2,0){$\bullet$}

\draw (-3,0) to [out=90, in=180] +(0.3,0.3);
\draw (-2.4,0) to [out=90, in=0] +(-0.3,0.3);

\draw (-3,1) to [out=-90, in=180] +(0.3,-0.3);
\draw (-2.4,1) to [out=-90, in=0] +(-0.3,-0.3);

\put(50,0){$,$}


\draw  (1.2,0) -- +(0,1);
\draw (1.7,0) to [out=65,in=-115] +(0.6,1);
\draw (2.3,0) to [out=115,in=-65] +(-0.6,1);

\draw (2.8,0) to [out=90, in=180] +(0.3,0.3);
\draw (3.4,0) to [out=90, in=0] +(-0.3,0.3);

\draw (2.8,1) to [out=-90, in=180] +(0.3,-0.3);
\draw (3.4,1) to [out=-90, in=0] +(-0.3,-0.3);

\put(139,2.2){$\bullet$}

\end{scope}
\end{tikzpicture}$$

Recall that the multiplication in $\mc G(d)\cong A_d$ is performed by concatenation of graphs as in \cite{KT}, and this is exactly the multiplication that we will use. Therefore it is natural to set  \\
$$
\begin{tikzpicture}[scale=0.7,thick,>=angle 90]
\begin{scope}[xshift=4cm]
\draw (2,0.5) node[] {${y}_j^2 := $};
\draw  (3,0) -- +(0,1);
\draw (5,0) -- +(0,1);
\draw [dotted] (3.55,.5) -- +(0.4,0);
\draw  (4.4,0) -- +(0,1);
\draw  (5.6,0) -- +(0,1);
\draw [dotted] (6.1,.5) -- +(0.4,0);
\draw  (7,0) -- +(0,1);
\put(177,25){\tiny{$j$}}
\put(177,-10){\tiny{$\overline{j}$}}
\put(176.5, 13){$\bullet$}
\put(176.5,6){$\bullet$}
\end{scope}
\end{tikzpicture}\\[3mm]
$$
and similarly for $y_j^k$ for any $k\in\mathbb{N}$. It is also clear that $y_iy_j=y_jy_i$ for $i\neq j$ since we allow the dots to freely move in a vertical line.

However, the concatenation of two dot Brauer $d$-diagrams may produce some graph which is not a dot Brauer $d$-diagram. Our next goal is to explain how to move a dot appearing in the middle of a graph until it is moved to the top or the bottom of a diagram, and this is performed by repeat using of the defining relations of $\Pd$.
Note that all of the defining relations of $\Pd$ in Definition \ref{def:VW} can be translated into the language of dot Brauer $d$-diagrams, where those relations without any $y_i$ have been already recorded in \cite{KT}. We list a few relations involving $y_i$ here as illustrating examples: \\ 

$(P.8)$ \begin{tikzpicture}[scale=0.7,thick,>=angle 90]
\begin{scope}[xshift=4cm]
\draw (0.8,0) to [out=90, in=180] +(0.3,0.3);
\draw (1.4,0) to [out=90, in=0] +(-0.3,0.3);
\draw (0.8,1) to [out=-90, in=180] +(0.3,-0.3);
\draw (1.4,1) to [out=-90, in=0] +(-0.3,-0.3);
\put(104,0){$\bullet$}
\put(123,0){$\bullet$}
\put(112,7){$=$}
\draw (2.3,0) to [out=90, in=180] +(0.3,0.3);
\draw (2.9,0) to [out=90, in=0] +(-0.3,0.3);
\draw (2.3,1) to [out=-90, in=180] +(0.3,-0.3);
\draw (2.9,1) to [out=-90, in=0] +(-0.3,-0.3);
\put(200,1){$,$}
\put(243.5,13.5){$\bullet$}
\put(263,13.5){$\bullet$}
\put(143,7){$+$}
\draw (3.8,0) to [out=90, in=180] +(0.3,0.3);
\draw (4.4,0) to [out=90, in=0] +(-0.3,0.3);
\draw (3.8,1) to [out=-90, in=180] +(0.3,-0.3);
\draw (4.4,1) to [out=-90, in=0] +(-0.3,-0.3);
\put(230,-35){$,$}

\draw (7.8,0) to [out=90, in=180] +(0.3,0.3);
\draw (8.4,0) to [out=90, in=0] +(-0.3,0.3);
\draw (7.8,1) to [out=-90, in=180] +(0.3,-0.3);
\draw (8.4,1) to [out=-90, in=0] +(-0.3,-0.3);
\put(252,7){$=$}
\draw (9.3,0) to [out=90, in=180] +(0.3,0.3);
\draw (9.9,0) to [out=90, in=0] +(-0.3,0.3);
\draw (9.3,1) to [out=-90, in=180] +(0.3,-0.3);
\draw (9.9,1) to [out=-90, in=0] +(-0.3,-0.3);
\put(283,7){$-$}
\draw (10.8,0) to [out=90, in=180] +(0.3,0.3);
\draw (11.4,0) to [out=90, in=0] +(-0.3,0.3);
\draw (10.8,1) to [out=-90, in=180] +(0.3,-0.3);
\draw (11.4,1) to [out=-90, in=0] +(-0.3,-0.3);
\end{scope} \end{tikzpicture}\\

$(P.7)$ \begin{tikzpicture}[scale=0.7,thick,>=angle 90]
\begin{scope}[xshift=4cm]

\draw (0.8,0) to [out=65,in=-115] +(0.6,1);
\draw (1.4,0) to [out=115,in=-65] +(-0.6,1);
\put(103,13.5){$\bullet$}
\put(111,7){$=$}
\draw (2.3,0) to [out=65,in=-115] +(0.6,1);
\draw (2.9,0) to [out=115,in=-65] +(-0.6,1);
\put(124,0){$\bullet$}
\put(143,7){$+$}
\draw (3.8,0) to [out=90, in=180] +(0.3,0.3);
\draw (4.4,0) to [out=90, in=0] +(-0.3,0.3);
\draw (3.8,1) to [out=-90, in=180] +(0.3,-0.3);
\draw (4.4,1) to [out=-90, in=0] +(-0.3,-0.3);
\put(173,7){$+$}
\draw (5.4,0) -- +(0,1);
\draw (5.9,0) -- +(0,1);


\draw (8.8,0) to [out=65,in=-115] +(0.6,1);
\draw (9.4,0) to [out=115,in=-65] +(-0.6,1);
\put(284,13.5){$\bullet$}
\put(272,7){$=$}
\draw (10.3,0) to [out=65,in=-115] +(0.6,1);
\draw (10.9,0) to [out=115,in=-65] +(-0.6,1);
\put(263,0){$\bullet$}
\put(303,7){$-$}
\draw (11.8,0) to [out=90, in=180] +(0.3,0.3);
\draw (12.4,0) to [out=90, in=0] +(-0.3,0.3);
\draw (11.8,1) to [out=-90, in=180] +(0.3,-0.3);
\draw (12.4,1) to [out=-90, in=0] +(-0.3,-0.3);
\put(333,7){$+$}
\draw (13.4,0) -- +(0,1);
\draw (13.9,0) -- +(0,1);
\end{scope} \end{tikzpicture}\\

Using the relations, we may move a dot in the middle to either the top or the bottom of the graph, where some terms with less dots might appear. For example, the concatenation of the following two graphs will have a dot in the middle:
$$ 
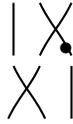
\begin{tikzpicture}[scale=0.7,thick,>=angle 90]
\begin{scope}[xshift=4cm]

\draw  (-4.6,0) -- +(0,1);
\draw (-4.1,0) to [out=65,in=-115] +(0.6,1);
\draw (-3.5,0) to [out=115,in=-65] +(-0.6,1);
\put(5.0,0){$\bullet$}

\draw (-4.7,-1.2) to [out=65,in=-115] +(0.6,1);
\draw (-4.0,-1.2) to [out=115,in=-65] +(-0.6,1);
\draw  (-3.5,-1.2) -- +(0,1);

\end{scope}
\end{tikzpicture}$$
But we can use the relation $(P.7)$ to replace the graph on the upper half. Using distribution law, the resulted concatenation equals to

$$ \begin{tikzpicture}[scale=0.7,thick,>=angle 90]
\begin{scope}[xshift=4cm]

\draw  (-4.6,0) -- +(0,1);
\draw (-4.1,0) to [out=65,in=-115] +(0.6,1);
\draw (-3.5,0) to [out=115,in=-65] +(-0.6,1);
\put(-4,15){$\bullet$}

\draw (-4.7,-1.2) to [out=65,in=-115] +(0.6,1);
\draw (-4.0,-1.2) to [out=115,in=-65] +(-0.6,1);
\draw  (-3.5,-1.2) -- +(0,1);

\put(20,-5){$-$}

\draw  (-2,0) -- +(0,1);
\draw (-1.5,0) to [out=90, in=180] +(0.3,0.3);
\draw (-.9,0) to [out=90, in=0] +(-0.3,0.3);
\draw (-1.5,1) to [out=-90, in=180] +(0.3,-0.3);
\draw (-0.9,1) to [out=-90, in=0] +(-0.3,-0.3);

\draw (-2.0,-1.2) to [out=65,in=-115] +(0.6,1);
\draw (-1.4,-1.2) to [out=115,in=-65] +(-0.6,1);
\draw  (-0.9,-1.2) -- +(0,1);

\put(80,-5){$+$}

\draw  (1,0) -- +(0,1);
\draw  (1.6,0) -- +(0,1);
\draw  (2.2,0) -- +(0,1);

\draw (1.0,-1.2) to [out=65,in=-115] +(0.6,1);
\draw (1.6,-1.2) to [out=115,in=-65] +(-0.6,1);
\draw  (2.2,-1.2) -- +(0,1);

\end{scope}
\end{tikzpicture}$$

We point out here that the dot in the first graph is located in the top, while the other two graphs have no dot anymore.

Let $\gamma$ be a dot Brauer $d$-diagram. Define the degree of $\gamma$ by
$$\deg \gamma \text{ = the number of dots in the graph } \gamma.$$
It gives a filtration on $\hat{\mc G}(d)$ and let $\gr\hat{\mc G}(d)$ be the associated graded algebra.
The next two lemmas follow from using the defining relations finitely many times and induction on degree, as illustrated by the example above.

\begin{lemma}[Lemma \ref{NazarovLemma44}]
If $\gamma$ is a concatenation of two dot Brauer $d$-diagrams with some dots in the middle, then $\gamma$ can be expressed as a linear combination of dot Brauer $d$-diagrams.
\end{lemma}

\begin{lemma}[Lemma \ref{NazarovLemma45}]
A non-regular dot Brauer d-diagram with degree $n$ can be expressed as a linear combination of regular ones with degree not greater than $n$.
\end{lemma}


\begin{cor} \label{Cor::GraphAlgebras}
The concatenation gives a well-defined multiplication on $\hat{\mc G}(d)$.
In particular, $\hat{\mc G}(d)\cong\Pd$ as associative algebras.
\end{cor}

\begin{rem}
	As in the case of \cite{KT}, if any circle (no matter how many dots attached to it) appears in a graph $\gamma$ resulted from the concatenation of two dot Brauer $d$-diagrams, then $\gamma=0$.
\end{rem}

We conclude this article by the following theorem, which is an immediate result of Theorem \ref{thm::PBW} and Corollary \ref{Cor::GraphAlgebras}.
\begin{thm} \label{Cor::DiagramInterpretation}
$\Pd$ can be realized as a diagram algebra with basis consisting of regular dot Brauer $d$-diagrams by the following identifications:\\
$$
\begin{tikzpicture}[scale=0.7,thick,>=angle 90]
\begin{scope}[xshift=4cm]
\draw (-6,0.5) node[] {${\varepsilon}_i\equiv$};
\draw  (-5,0) -- +(0,1);
\draw [dotted] (-4.65,.5) -- +(0.4,0);
\draw  (-4,0) -- +(0,1);
\draw (-3.2,0) to [out=90, in=180] +(0.3,0.3);
\draw (-2.6,0) to [out=90, in=0] +(-0.3,0.3);
\draw (-3.2,1) to [out=-90, in=180] +(0.3,-0.3);
\draw (-2.6,1) to [out=-90, in=0] +(-0.3,-0.3);
\draw  (-1.8,0) -- +(0,1);
\draw [dotted] (-1.5,.5) -- +(0.4,0);
\draw  (-0.85,0) -- +(0,1);
\put(12,22){\tiny{$i$}}\put(22,22){\tiny{$i+1$}}
\put(12,-10){\tiny{$\overline{i}$}}\put(22,-10){\tiny{$\overline{i+1}$}}

\draw (2,0.5) node[] {${s}_i\equiv$};
\draw  (3,0) -- +(0,1);
\draw (4.6,0) to [out=65,in=-115] +(0.6,1);
\draw (5.2,0) to [out=115,in=-65] +(-0.6,1);
\draw [dotted] (3.35,.5) -- +(0.4,0);
\draw  (4,0) -- +(0,1);
\draw  (6,0) -- +(0,1);
\draw [dotted] (6.3,.5) -- +(0.4,0);
\draw  (7,0) -- +(0,1);
\put(167,22){\tiny{$i$}}\put(178,22){\tiny{$i+1$}}
\put(167,-10){\tiny{$\overline{i}$}}\put(178,-10){\tiny{$\overline{i+1}$}}

\end{scope}
\end{tikzpicture}
$$

$$
\begin{tikzpicture}[scale=0.7,thick,>=angle 90]
\begin{scope}[xshift=4cm]
\draw (2,0.5) node[] {${y}_i\equiv$};
\draw  (3,0) -- +(0,1);
\draw (5,0) -- +(0,1);
\draw [dotted] (3.55,.5) -- +(0.4,0);
\draw  (4.4,0) -- +(0,1);
\draw  (5.6,0) -- +(0,1);
\draw [dotted] (6.1,.5) -- +(0.4,0);
\draw  (7,0) -- +(0,1);
\put(177,22){\tiny{$i$}}
\put(177,-10){\tiny{$\overline{i}$}}
\put(176.5, 13){$\bullet$}

\end{scope}
\end{tikzpicture}
$$
\\
The regular monomials correspond to regular dot Brauer $d$-diagrams under the identification above.
\end{thm}

\appendix
\section{}
\subsection{} We need some preparations to prove the relations $(P.8)$ in Proposition \ref{not4}.
Note that we may assume the relations $(P.6)$ hold in $\End_\frak{g}(M\otimes V^{\otimes d})$, since their proofs do not involve $(P.8)$ or their consequences. For convenience, set $\ell=\dim \mf g$ throughout the appendix.
\begin{lemma} \label{lem::BasicLemForCalculation} For each $m\in M\otimes V^{\otimes i-1}$ and $w \in V^{\otimes d-(i+1)}$, we have
\begin{align}
   &\sum_{k\in I}(-1)^{\ov{e_k}}m\otimes \left( x_j e_{k} \otimes e_{\bar{k}} + (-1)^{\ov{x_j}\ov{e_k}}  e_{k} \otimes x_j e_{\bar{k}}\right) \otimes w =0 , \label{Eq::lem::BasicLemForCalculation-Eq1} \\
   &\vep_i\left(m \otimes( x_j e_{a} \otimes e_b + (-1)^{\ov{x_j}\ov{e_a}} e_{a} \otimes x_j e_b)\otimes w  \right)=0, \label{Eq::lem::BasicLemForCalculation-Eq2}
\end{align} for all $1\leq j \leq  \ell$ and $1\leq a,b \leq n$.
\end{lemma}
\begin{proof}
Observe that $\vep(V\otimes V)$ is a trivial $\mf g$-module, and \eqref{Eq::lem::BasicLemForCalculation-Eq1} follows.

Consider \eqref{Eq::lem::BasicLemForCalculation-Eq2}.
    \begin{align}
    &\vep_i\left(m \otimes( x_j e_{a} \otimes e_b + (-1)^{\ov{x_j}\ov{e_a}} e_{a} \otimes x_j e_b)\otimes w  \right) \\
    &= \vep_i\left(m \otimes \big( (x_j e_{a},e_b )e_\barb \otimes e_b + (-1)^{\ov{x_j}\ov{e_a}} e_{a} \otimes (x_j e_b,e_a)e_\bara \big)\otimes w  \right)\\
    &= m \otimes \left( \big( (x_j e_{a},e_b )  + (-1)^{\ov{x_j}\ov{e_a}} (x_j e_b,e_a) \big) \sum_{k\in I}(-1)^{\ov{e_k}}e_k \otimes e_{\bar{k}} \right) \otimes w =0.   \label{Eq::lem::BasicLemForCalculation::ProofEq2}
  \end{align}
      The  equation \eqref{Eq::lem::BasicLemForCalculation::ProofEq2} follows from the following two facts \begin{itemize}
      \item[(1)] $(-1)^{\ov{x_j}+\ov{e_b}} = (-1)^{\ov{e_{a}}+1}$ and so $(-1)^{\ov{x_j}\ov{e_b}} = (-1)^{\ov{x_j}\ov{e_{a}}}$, if $(x_je_{b}, e_a)\neq 0$.
      \item[(2)] $(x_je_{{b}},e_{a})  = -(-1)^{\ov{x_j}\ov{e_b}}(x_je_{a}, e_{b})$.
    \end{itemize}
\end{proof}

The following lemma is an analogue of \cite[Lemma 8.4]{ES}.
\begin{lemma} \label{lem::ComLem1}
  For $k>i+1$ and $i>0$, we have \begin{align}\label{Eq::eAndOmegaEq1}(\Omega_{i,k}+\Omega_{i+1,k})\vep_i = 0 = \vep_i(\Omega_{i,k}+\Omega_{i+1,k}). \end{align}
\end{lemma}
\begin{proof}
   We first prove \eqref{Eq::eAndOmegaEq1} for $i=1$ and $k=3$. For all $a,b,c\in I$, we may observe that
   \begin{align*}
     &\big((\Omega_{1,3}+\Omega_{2,3})\vep_1\big)(m\otimes e_a\otimes e_\bara \otimes e_c) \\
     &=  (\Omega_{1,3}+\Omega_{2,3}) \sum_{k\in I} (-1)^{\ov{e_k}}m\otimes e_k\otimes e_{\ov{k}}\otimes e_c \\
     &=  \sum_{j=1}^{\ell}\sum_{k\in I} (-1)^{\ov{e_k} +\ov{x_j} }m\otimes x_je_k\otimes e_{\ov{k}}\otimes x_j^*e_c + \sum_{j=1}^{\ell}\sum_{k\in I} (-1)^{\ov{e_k}+\ov{x_j}+\ov{x_j} \ov{e}_k}m\otimes e_k\otimes x_je_{\ov{k}}\otimes x_j^*e_c \\ &=  \sum_{j=1}^{\ell} \sum_{k\in I} (-1)^{\ov{e_k} +\ov{x_j} }m\otimes \left(x_je_k\otimes e_{\ov{k}} + (-1)^{\ov{x_j} \ov{e}_k} e_k\otimes x_je_{\ov{k}}\right)\otimes x_j^*e_c, \\
   \end{align*} which is zero  by Lemma \ref{lem::BasicLemForCalculation}

     We now consider the right hand side of identity \eqref{Eq::eAndOmegaEq1} as follows: \begin{align*}
     &\big(\vep_1(\Omega_{1,3}+\Omega_{2,3})\big)(m\otimes e_a\otimes e_b \otimes e_c) \\
     &=  \vep_1 \left(\sum_{j=1}^{\ell}  m\otimes \left( x_je_a\otimes e_{b}+ (-1)^{\ov{e_a}\ov{x_j}}e_a\otimes x_je_b\right)\otimes (-1)^{\ov{x_j}(\ov{e_a}+\ov{e_b})}x_j^*e_c\right),
   \end{align*} which is zero by Lemma \ref{lem::BasicLemForCalculation}.

   Similar calculations hold for general $i$ and $k$ by using Lemma \ref{lem::BasicLemForCalculation}. This completes the proof.
\end{proof}

\subsection{Proof of equation $(P.8)$(a)} \label{PfOf8a} In this subsection, we prove the equation (P.8)(b).
\begin{lemma} \label{lem::Eq::Cal::VW82} We have
 \begin{align}\label{Eq::Cal::VW82} &\sum_{k\in I}(-1)^{\ov{e_{k}}} x^*_ie_k\otimes e_{\bar{k}} - \sum_{k\in I}(-1)^{\ov{e_{k}}+\ov{x_i}\ov{e_k}}e_k\otimes x^*_ie_{\bar{k}} = 0 ,\end{align}
 for all $1\leq i\leq \ell$.
\end{lemma}
\begin{proof}
The proof is completed by the following calculations.
\begin{itemize}
  \item[(1)] Set $s,t\in I^0$  and $x_i^*:=E_{ts}^*$ in \eqref{Eq::Cal::VW82}, we obtain  $$ \frac{1}{2}\left((e_s \otimes e_{\ov t} - e_{\ov t}\otimes e_s) -(-e_{\ov t}\otimes e_s +e_s \otimes e_{\ov t})\right) =0.$$
  \item[(2)] Set $s,t\in I^0$ with $s\neq t$,  and $x_i^*:=Y_{st}^*$ in \eqref{Eq::Cal::VW82}, we obtain  $$ \frac{-1}{2}\left((-e_s \otimes e_{t} - e_{t}\otimes e_s) -(-e_{t}\otimes e_s -e_s \otimes e_{t})\right) =0.$$
  \item[(3)] Set $s,t\in I^0$ with $s\neq t$,  and $x_i^*:=X_{st}^*$ in \eqref{Eq::Cal::VW82}, we obtain  $$ \frac{-1}{2}\left((e_{\ov s} \otimes e_{\ov t} -(-1)(-1) e_{\ov t}\otimes e_{\ov s}) +(e_{\ov t}\otimes e_{\ov s} - e_{\ov s} \otimes e_{\ov t})\right) =0.$$
  \item[(4)] Set $s\in I^0$ and $x_i^*:=X_{ss}^*$ in \eqref{Eq::Cal::VW82}, we obtain  $$ -(e_{\ov s} \otimes e_{\ov s} -(-1)(-1) e_{\ov s}\otimes e_{\ov s}) =0.$$
\end{itemize}
\end{proof}

\begin{lemma} \emph{(}Equation  $(P.8)$\emph{(}a\emph{)}\emph{)}
  $\vep_a (y_a- y_{a+1}) = \vep_a$, for all $1\leq a\leq d-1$.
\end{lemma}
\begin{proof}
   Let $m\in M\otimes V^{\otimes a-1}$ and $v\in V^{\otimes d-a-1}$ be homogenous elements. Then for all $i,j \in I$ we have the following calculation.
   \begin{align}
   &\big((y_a- y_{a+1})\vep_a\big)(m\otimes e_i \otimes e_j \otimes v)  \\
   & = 2\delta_{i,\ov j}(y_a -y_{a+1}) \left(m\otimes \sum_{k\in I} (-1)^{\ov{e_k}} e_k \otimes e_{\ov k} \otimes v \right) \\
   &= 2\delta_{i,\ov j}\Big(\sum_{q=1}^{\ell} x_qm\otimes \sum_{k\in I} (-1)^{\ov{e_k}+\ov{x_q}\ov m} x^*_qe_k \otimes e_{\ov k} \otimes v - \sum_{q=1}^{\ell} x_qm\otimes \sum_{k\in I} (-1)^{\ov{e_k}+ \ov{x_q}(\ov m+ \ov{e_k})} e_k \otimes x_q^*e_{\ov k} \otimes v \\
   &-\sum_{q=1}^{\ell} m\otimes \sum_{k\in I} (-1)^{\ov{e_k}+\ov{x_q}\ov{e_k}} x_qe_k \otimes x_q^*e_{\ov k} \otimes v \Big) =  -2\delta_{i,\ov j}\sum_{q=1}^{\ell} m\otimes \sum_{k\in I} (-1)^{\ov{e_k}+\ov{x_q}\ov{e_k}} x_qe_k \otimes x_q^*e_{\ov k} \otimes v.
   \end{align} The last equation follows from Lemma \ref{lem::Eq::Cal::VW82}. We now observe that the last term above can be rewritten as follows:
   \begin{align}
   &-2\delta_{i,\ov j}\sum_{q=1}^{\ell} m\otimes \sum_{k\in I} (-1)^{\ov{e_k}+\ov{x_q}\ov{e_k}} x_qe_k \otimes x_q^*e_{\ov k} \otimes v = - \delta_{i,\ov j} \,m\otimes  \Big(2C_{V,V}\big(\vep(e_i \otimes e_{j})\big)\Big) \otimes v.
   \end{align}
   By definition of $\vep$ in \eqref{def::vep} and $(P.6)$(a), we have $2C_{V,V}\circ \vep = (\vep +s)\circ \vep =s\circ \vep = -\vep$ in $\text{End}_{\G}(V^{\otimes 2})$. This completes the proof.
\end{proof}

\subsection{Proof of equation (P.8)(b)}\label{PfOf8b} In this subsection, we prove the equation (P.8)(b).
\begin{lemma} \label{lem::Eq::Cal::VW81} We have
 \begin{align}\label{Eq::Cal::VW81} &\vep(x_i^*e_p \otimes e_q - (-1)^{\ov{x_i}\ov{e_p}}e_p\otimes x_i^*e_q)= 0 ,\end{align}
 for all $1\leq i\leq \ell$ and $p,q \in I$.
\end{lemma}
\begin{proof} Recall that $\vep(e_p\otimes e_q) =0$ if $p\neq \ov q$. Therefore, the proof reduces to the following calculations:
\begin{itemize}
  \item[(1)] Set $s\in I^0$ and $x_i^*:=E_{ss}^*$ in \eqref{Eq::Cal::VW81}, and assume that $p=s$, $q=\ov j$ for some $j\in I^0$ with $j\neq s$. Then we obtain  $$ \frac{1}{2}\vep\left( e_s\otimes e_{\ov j} \right)= 0.$$
  \item[(2)] Set $s\in I^0$ and $x_i^*:=E_{ss}^*$ in \eqref{Eq::Cal::VW81}, and assume that $p=j$, $q=\ov s$ for some $j\in I^0$ with $j\neq s$. Then we obtain  $$ \frac{1}{2}\vep\left( e_j\otimes e_{\ov s} \right)= 0.$$
  \item[(3)] Set $s\in I^0$ and $x_i^*:=E_{ss}^*$ in \eqref{Eq::Cal::VW81}, and assume that $p=s$, $q=\ov s$ (resp. $p=\ov  s$, $q=s$). Then we obtain  $$\frac{1}{2} \vep\left( e_s\otimes e_{\ov s} -   e_s\otimes e_{\ov s} \right)= \vep(0) = 0.$$  $$(\text{resp. } \vep\left( e_{\ov s}\otimes e_{ s} -   e_{\ov s}\otimes e_{s} \right)= \vep(0) = 0)$$
  \item[(4)] Set $s,t\in I^0$, $s\neq t$ and $x_i^*:=E_{ts}^*$ in \eqref{Eq::Cal::VW81}, and assume that $p=t$, $q=\ov j$. Then we obtain  $$ \frac{1}{2}\vep\left( e_s\otimes e_{\ov j} -  (\delta_{j,s} e_t\otimes e_{\ov t}) \right)= 0.$$
   \item[(5)] Set $s,t\in I^0$, $s\neq t$ and $x_i^*:=E_{ts}^*$ in \eqref{Eq::Cal::VW81}, and assume that $p=\ov s$, $q= j$. Then we obtain  $$\frac{1}{2} \vep\left( e_{\ov t}\otimes e_{j} -  (\delta_{j,t} e_{\ov s}\otimes e_{s}) \right)= 0.$$
   \item[(6)] Set $s,t\in I^0$, $s\neq t$ and $x_i^*:=E_{ts}^*$ in \eqref{Eq::Cal::VW81}, and assume that $p= j$, $q= \ov s$. Then we obtain  $$ \frac{1}{2}\vep\left( \delta_{j,t}e_{s}\otimes e_{\ov s} -  e_{ j}\otimes e_{\ov t} \right)= 0.$$
   \item[(7)] Set $s,t\in I^0$, $s\neq t$ and $x_i^*:=E_{ts}^*$ in \eqref{Eq::Cal::VW81}, and assume that $p= \ov j$, $q= t$. Then we obtain  $$ \frac{1}{2}\vep\left( \delta_{j,s}e_{\ov t}\otimes e_{t} -  e_{\ov j}\otimes e_{s} \right)= 0.$$
   \item[(8)] Set $s,t\in I^0$, $s\neq t$ and $x_i^*:=Y_{st}^*$ in \eqref{Eq::Cal::VW81}, and assume that $p, q \in I\backslash I^0$. Then we obtain  \begin{align*} & -\frac{1}{2}\left(\vep\left( \delta_{\ov p, t}e_{s}\otimes e_{q} - (-1) e_{ p}\otimes \delta_{\ov q,t} e_{s} \right) -\vep\left(\delta_{\ov p, s}e_t\otimes e_q -(-1)e_p\otimes \delta_{\ov q, s}e_t \right)\right) \\ &= -\frac{1}{2}\left(\delta_{\ov p,s}\delta_{\ov q,t} - \delta_{t,\ov q}\delta_{s, \ov p}\right)\sum_{k\in I}(-1)^{\ov{e_k}}e_k \otimes e_{\ov k} = 0.\end{align*}
   \item[(9)] Set $s,t\in I^0$, $s\neq t$ and $x_i^*:=X_{st}^*$ in \eqref{Eq::Cal::VW81}, and assume that $p, q \in I^0$. Then we obtain  \begin{align*} & -\frac{1}{2}\left(\vep\left( \delta_{p, t}e_{\ov s}\otimes e_{q} - e_{p}\otimes \delta_{q,t} e_{\ov s} \right) +\vep\left(\delta_{p, s}e_{\ov t}\otimes e_q -e_p\otimes \delta_{q, s}e_{\ov t} \right)\right) \\ &= -\frac{1}{2}\left(\delta_{q,s}\delta_{p,t} - \delta_{s,p}\delta_{t, q}+ \delta_{p,s}\delta_{q,t} - \delta_{p,t}\delta_{q,s}\right)\sum_{k\in I}(-1)^{\ov{e_k}}e_k \otimes e_{\ov k} = 0.\end{align*}
   \item[(10)] Set $s\in I^0$ and $x_i^*:=X_{ss}^*$ in \eqref{Eq::Cal::VW81}, and assume that $p, q \in I^0$. Then we obtain  \begin{align*} &-\vep\left( \delta_{p, s}e_{\ov p}\otimes e_{q} - \delta_{q,s}e_{p}\otimes e_{\ov q} \right) =-\left(\delta_{p,s}\delta_{p,q} - \delta_{q,s}\delta_{p, q}\right)\sum_{k\in I}(-1)^{\ov{e_k}}e_k \otimes e_{\ov k} = 0.\end{align*}

\end{itemize}
\end{proof}

\begin{lemma} \emph{(}Equation  $(P.8)$\emph{(}b\emph{)}\emph{)}
  $(y_a- y_{a+1})\vep_a =  -\vep_a$, for all $1\leq a\leq d-1$.
\end{lemma}
\begin{proof}
 Let $m\in M\otimes V^{\otimes a-1}$ and $v\in V^{\otimes d-a-1}$ be homogenous elements. Then for all $i,j \in I$ we have the following calculation.
   \begin{align}
   &\big(\vep_a(y_a- y_{a+1})\big)(m\otimes e_i \otimes e_j \otimes v)  \\ &=2\vep_a( \sum_{k=1}^{\ell} x_k m\otimes (-1)^{\ov{x_k}\ov m} x_k^*e_i \otimes e_j \otimes v - \sum_{k=1}^{\ell} (-1)^{\ov{x_k}(\ov m+\ov{e_i})}x_k m\otimes e_i \otimes x_k^*e_j \otimes v  \\ &-\sum_{k=1}^{\ell} (-1)^{\ov{x_k}\ov{e_i}} m\otimes x_ke_i \otimes x_k^*e_j \otimes v) \\ & =-2\vep_a \left(\sum_{k=1}^{\ell} (-1)^{\ov{x_k}\ov{e_i}} m\otimes x_ke_i \otimes x_k^*e_j \otimes v\right).
   \end{align} The last equation follows from Lemma \ref{lem::Eq::Cal::VW81}. We now observe that the last term above can be rewritten as follows:
   \begin{align}
   &-2\vep_a \left(\sum_{k=1}^{\ell} (-1)^{\ov{x_k}\ov{e_i}} m\otimes x_ke_i \otimes x_k^*e_j \otimes v\right) = - m\otimes  \vep\big(2 C_{V,V}(e_i \otimes e_j)\big) \otimes v.
   \end{align} By definition of $\vep$ in \eqref{def::vep} and $(P.6)$(a), we have $\vep\circ 2C_{V,V} = \vep \circ (\vep +s) =\vep \circ s = \vep$ in $\text{End}_{\G}(V^{\otimes 2})$. This completes the proof.
\end{proof}

\vspace{5mm}


\end{document}